\documentclass[3p,times,nonatbib]{elsarticle}
\usepackage{amsmath,amssymb,amsthm,mathtools,bm,aliascnt}
\allowdisplaybreaks
\usepackage{graphicx,booktabs,array,microtype,enumitem}
\usepackage[backend=biber,style=numeric,sorting=nyt,maxbibnames=99]{biblatex}
\usepackage[hidelinks]{hyperref}
\usepackage[nameinlink,capitalise]{cleveref}
\crefname{theorem}{theorem}{theorems}\Crefname{theorem}{Theorem}{Theorems}
\crefname{proposition}{proposition}{propositions}\Crefname{proposition}{Proposition}{Propositions}
\crefname{lemma}{lemma}{lemmas}\Crefname{lemma}{Lemma}{Lemmas}
\crefname{corollary}{corollary}{corollaries}\Crefname{corollary}{Corollary}{Corollaries}
\crefname{remark}{remark}{remarks}\Crefname{remark}{Remark}{Remarks}
\usepackage{caption}
\graphicspath{{figures/}}

\newtheorem{theorem}{Theorem}[section]
\newaliascnt{proposition}{theorem}
\newtheorem{proposition}[proposition]{Proposition}
\aliascntresetthe{proposition}
\newaliascnt{lemma}{theorem}
\newtheorem{lemma}[lemma]{Lemma}
\aliascntresetthe{lemma}
\newaliascnt{corollary}{theorem}
\newtheorem{corollary}[corollary]{Corollary}
\aliascntresetthe{corollary}
\theoremstyle{remark}
\newaliascnt{remark}{theorem}
\newtheorem{remark}[remark]{Remark}
\aliascntresetthe{remark}
\newaliascnt{definition}{theorem}

\aliascntresetthe{definition}

\newcommand{\PP}{\mathbb P}
\newcommand{\C}{\mathbb C}

\newcommand{\Patches}{\mathcal P}
\newcommand{\F}{\mathcal F}
\newcommand{\Lop}{\mathcal L_{A,\eta,\kappa}}

\newcommand{\norm}[1]{\left\lVert #1\right\rVert}
\newcommand{\jump}[1]{\left[\!\left[#1\right]\!\right]}

\journal{Computers \& Mathematics with Applications}

\begin{document}
\begin{frontmatter}

\title{Efficient $C^1$ Bernstein Quasi-Trefftz Discretization for Heterogeneous High-Frequency Helmholtz Problems}

\author[tamu]{Shelvean Kapita}
\affiliation[tamu]{organization={Department of Mathematics, Texas A\&M University},
            city={College Station},
            state={Texas},
            country={USA}}

\begin{abstract}
We develop an efficient $C^1$ Bernstein quasi-Trefftz discretization for heterogeneous high-frequency Helmholtz problems with smooth variable coefficients. The method starts from ordinary Bernstein--B\'ezier polynomial pieces on small unstructured macro-patches. Classical B-coefficient smoothness relations are imposed first, producing an exact local $C^1$ space, and the differential equation is then used to remove locally resolvable residual directions. On two-triangle patches this reduces a quadratic-size conforming polynomial space to a trace-sized equation-adapted space before global assembly. Distinct patches are coupled through value and conormal-flux mismatches together with the remaining cell residual and physical boundary residuals. Nonpolynomial coefficients are approximated only during local space construction; the global residual is evaluated with the original medium. We establish residual best approximation, quantify the perturbation induced by coefficient projection, and derive a physical-energy reliability estimate with explicit dependence on the high-frequency resolution parameter. An optimized implementation uses explicit $C^1$ continuation, batched coefficient projection, batched residual formation, and reduced-coordinate assembly. Numerical experiments on noncongruent unstructured meshes demonstrate roundoff-level $C^1$ conformity, strong local compression, high-order resolution, controlled pollution, exact-circle penetrable scattering with a Fourier NtD map, and robust behavior through an Airy turning point.
\end{abstract}


\end{frontmatter}

\section{Introduction}\label{sec:intro}
High-frequency Helmholtz discretizations have to resolve oscillation while keeping approximation spaces, linear systems, and local basis construction numerically manageable. Trefftz methods address the first part of this problem by building the differential equation into the trial space, most commonly through plane waves, Fourier--Bessel functions, or related local wave solutions \cite{HiptmairMoiolaPerugiaSurvey2016}. Their effectiveness at high frequency is well established, but the local wave representation then becomes part of the numerical design: directions or special functions must be selected, scaled, and kept linearly independent \cite{GittelsonHiptmairPerugia2009,HiptmairMoiolaPerugia2011,ParolinHuybrechsMoiola2023}. Standard polynomial finite elements avoid that basis-selection problem but must resolve both oscillation and pollution effects \cite{BabuskaSauter1997,MelenkSauter2010,MelenkSauter2011}.

This paper develops a different route. We retain ordinary Bernstein--B\'ezier polynomials as the computational representation and extract a small equation-adapted subspace algebraically. The construction is motivated by two requirements that are especially useful for heterogeneous problems: the local representation should support exact geometric smoothness constraints, and the material coefficients should enter through local algebra rather than through a special-function basis. Bernstein coordinates satisfy both requirements. Differentiation, restriction to an edge, normal differentiation, and degree elevation are sparse coefficient operations, while classical spline theory provides explicit $C^1$ matching relations between adjacent triangles \cite{LaiSchumaker2007}.

The central design principle is \emph{conformity first, quasi-Trefftz second}. Each macro-patch is first reduced from broken Bernstein pieces to the exact strongly $C^1$ spline space allowed by its geometry. The Helmholtz residual is then imposed only inside that conforming space. For a full-row-rank two-triangle patch of degree $p$, this takes the local dimension from $p^2+p+1$ conforming coordinates to $2p+1$ equation-adapted coordinates. The compression is completed before neighboring patches interact. Global coupling therefore acts on trace-sized local vectors rather than on the full polynomial B-net.

This ordering is important computationally. The first reduction depends only on geometry and can be formed by explicit Bernstein continuation. The second reduction contains the differential operator, forcing, and heterogeneous coefficients and is obtained from a small residual-moment map. The third stage couples macro-patches through value and conormal-flux traces, the remaining cell residual, and the physical boundary residual. The resulting architecture separates geometry, local PDE adaptation, and global communication.

The method belongs to the broader family of quasi-Trefftz and embedded Trefftz discretizations. Generalized-plane-wave and Taylor constructions enforce the PDE to finite local order \cite{ImbertGerard2021,ImbertGerardEtAl2025,ImbertGerard2025Taylor}, while embedded Trefftz methods extract equation-adapted polynomial spaces through algebraic constraints \cite{LehrenfeldStocker2023,StockerVoulis2026}. Here the reduction is performed after imposing strong $C^1$ compatibility, and it is expressed directly in Bernstein coordinates. This makes the same local representation usable for smoothness, heterogeneous coefficients, nonzero forcing, local compression, and reduced global assembly.

The paper has four main computational contributions. First, we derive an explicit $C^1$ Bernstein synthesis on two-triangle patches and combine it with a stable QR-based quasi-Trefftz reduction. Second, we isolate the use of projected coefficients to local basis construction; the global residual continues to use the original medium, so coefficient approximation and discretization error remain separate. Third, we derive graph-residual and physical-energy estimates that expose the role of the high-frequency resolution parameter and the residual components that survive the local reduction. Fourth, we implement the method with batched coefficient projection, batched residual construction, residual-complement testing, and reduced-coordinate sparse assembly. In the exact-circle penetrable-scattering test with 200 macro-patches and 3400 active unknowns, optimized preprocessing takes less than one second, followed by a few seconds of global assembly and solution.

The numerical study is designed around computational questions rather than a single manufactured example. We verify strong $C^1$ conformity on noncongruent patches, examine conditioning of the reduced operator, measure convergence for a matrix-valued heterogeneous medium, identify the degree-resolution threshold on a fixed mesh, separate resolved and polluted high-frequency regimes at fixed $\kappa h/p$, test sensitivity to local coefficient projection, impose an exact Fourier NtD map on a circular truncation boundary for penetrable scattering, and finally stress the local reduction with an Airy turning-point transition.

The remainder of the paper is organized as follows. Section~\ref{sec:bern} develops the Bernstein--B\'ezier $C^1$ patch space. Section~\ref{sec:qt} constructs the local quasi-Trefftz reduction and analyzes coefficient perturbations. Section~\ref{sec:graph} gives the reduced global coupling and error estimates. Section~\ref{sec:impl} describes the implementation and computational savings. Section~\ref{sec:num} reports the numerical experiments. Section~\ref{sec:related} relates the method to Trefftz approaches and discusses computational implications and extensions.

\section{Bernstein--B\'ezier construction of the local \texorpdfstring{$C^1$}{C1} space}\label{sec:bern}
Let $P$ be a polygonal macro-patch equipped with a conforming triangulation $\Delta_P$, and write $N_T=|\Delta_P|$ for its number of triangles.  For a triangle
\[
 K=\langle v_1,v_2,v_3\rangle\in\Delta_P
\]
with barycentric coordinates $\lambda_1,\lambda_2,\lambda_3$, the Bernstein polynomial of degree $p$ associated with the multi-index $\alpha=(i,j,k)$, $|\alpha|=i+j+k=p$, is
\begin{equation}\label{eq:bernstein}
 B^p_{ijk}=\frac{p!}{i!j!k!}\lambda_1^i\lambda_2^j\lambda_3^k.
\end{equation}
There are
\[
 N_p=\dim\PP_p(K)=\frac{(p+1)(p+2)}2
\]
such polynomials.  We write the B-form of $s\in\PP_p(K)$ as
\begin{equation}\label{eq:bform}
 s|_K=\sum_{i+j+k=p}c^K_{ijk}B^p_{ijk}.
\end{equation}
Following the Bernstein--B\'ezier terminology of Lai and Schumaker \cite[Chapters 2--3]{LaiSchumaker2007}, the numbers $c^K_{ijk}$ are the \emph{B-coefficients}.  The associated domain point is
\[
 \xi^K_{ijk}=\frac{i v_1+j v_2+k v_3}{p},
\]
and the collection of the B-coefficients over the domain points is the B-net.  We denote by $c_K\in\C^{N_p}$ the vector of B-coefficients in a fixed ordering.

The local smoothness construction is carried out entirely in these ordinary polynomial coordinates.  First the B-coefficients are constrained so that the piecewise polynomial belongs to the classical spline space
\begin{equation}\label{eq:Spdef}
 S_p^1(\Delta_P)
 :=\{s\in C^1(P): s|_K\in\PP_p(K)\ \text{for every }K\in\Delta_P\}.
\end{equation}
Only after this exact $C^1$ space has been formed is the differential equation used to define the quasi-Trefftz reduction.

\subsection{Smoothness conditions across one interior edge}\label{subsec:c1edge}
Consider two triangles
\[
 K=\langle v_1,v_2,v_3\rangle,
 \qquad
 \widetilde K=\langle v_1,v_2,\widetilde v_3\rangle
\]
sharing the edge $e=\langle v_1,v_2\rangle$.  Write
\[
 s|_K=\sum_{i+j+k=p}c_{ijk}B^p_{ijk},
 \qquad
 s|_{\widetilde K}=\sum_{i+j+k=p}\widetilde c_{ijk}\widetilde B^p_{ijk},
\]
where the vertices on the common edge are ordered consistently.  Since $\lambda_3=0$ on $e$, the restriction of $s|_K$ to $e$ is determined exactly by the row of B-coefficients
\[
 c_{p,0,0},\ c_{p-1,1,0},\ldots,c_{0,p,0}.
\]
Hence continuity across $e$ is equivalent to equality of the B-coefficients associated with the domain points on the common edge,
\begin{equation}\label{eq:c0B}
 c_{ij0}=\widetilde c_{ij0},
 \qquad i+j=p.
\end{equation}
These are $p+1$ independent conditions for an isolated two-triangle join.

For the first-order smoothness condition, express the vertex $\widetilde v_3$ in barycentric coordinates with respect to $K$,
\begin{equation}\label{eq:oppbary}
 \widetilde v_3=\beta_1v_1+\beta_2v_2+\beta_3v_3,
 \qquad
 \beta_1+\beta_2+\beta_3=1.
\end{equation}
The standard Bernstein--B\'ezier $C^1$ smoothness condition across $e$ is then
\begin{equation}\label{eq:c1B}
 \boxed{\;
 \widetilde c_{ij1}
 =\beta_1c_{i+1,j,0}
  +\beta_2c_{i,j+1,0}
  +\beta_3c_{ij1},
 \qquad i+j=p-1.
 \;}
\end{equation}
Thus the first row of B-coefficients adjacent to $e$ on one triangle is determined by the edge row and the first adjacent row on the other triangle.  Equation \eqref{eq:c1B} is the degree-$p$ specialization of the classical Bernstein--B\'ezier smoothness relations; it supplies $p$ additional equations.  Geometrically, $C^0$ smoothness identifies the B-net on the common edge, while $C^1$ smoothness imposes the corresponding first-order relation on the two rows of domain points immediately adjacent to that edge.

For completeness, the same conditions can be written in a coordinate-free derivative form that is convenient later in the implementation.  Let $t_e$ be a tangent to $e$ and fix one physical normal $n_e$ for the geometric edge.  Equality of the traces in \eqref{eq:c0B} already implies equality of the tangential derivatives, since both pieces restrict to the same univariate degree-$p$ polynomial on $e$.  Therefore \eqref{eq:c0B}--\eqref{eq:c1B} are equivalently expressed as
\begin{equation}\label{eq:c1traceform}
 s|_{K}=s|_{\widetilde K}\quad\text{on }e,
 \qquad
 D_{n_e}s|_{K}=D_{n_e}s|_{\widetilde K}\quad\text{on }e.
\end{equation}
Because $\{t_e,n_e\}$ spans $\mathbb R^2$, these two identities imply equality of the full gradients on $e$ and hence a genuine $C^1$ join.

\subsection{Derivative and trace maps in B-coordinates}\label{subsec:berncalc}
Let $a\in\mathbb R^2$ and let $D_a=a\cdot\nabla$.  Since each barycentric coordinate is affine on $K$, the quantities
\[
 \alpha_r^K(a):=D_a\lambda_r,
 \qquad r=1,2,3,
\]
are constants and satisfy $\alpha_1^K(a)+\alpha_2^K(a)+\alpha_3^K(a)=0$.  Differentiation of \eqref{eq:bernstein} gives
\begin{equation}\label{eq:firstder}
 D_a B^p_{ijk}
 =p\sum_{r=1}^3\alpha_r^K(a)
 B^{p-1}_{(i,j,k)-e_r},
\end{equation}
with terms of negative index omitted.  Hence directional differentiation is a sparse operation on neighboring B-coefficients and lowers the degree by one.

The Cartesian derivatives are the special cases $a=e_1$ and $a=e_2$, written $D_x$ and $D_y$.  Applying \eqref{eq:firstder} twice gives the sparse coefficient map
\[
 D_K:\C^{N_p}\longrightarrow\C^{N_{p-2}}
\]
for $-\Delta s$.  Similarly, for an edge $e$ we denote by
\[
 T_{K,e}:\C^{N_p}\to\C^{p+1},
 \qquad
 N_{K,e}:\C^{N_p}\to\C^{p}
\]
the maps returning, respectively, the univariate Bernstein coefficients of $s|_e$ and $D_{n_e}s|_e$.  If $e=\langle v_1,v_2\rangle$ and $n_e$ is the fixed edge normal, then the $m$th degree-$(p-1)$ B-coefficient of the normal derivative, $m=0,\ldots,p-1$, is
\begin{equation}\label{eq:normalBcoeff}
 p\Big[
 \alpha_1^K(n_e)c_{p-m,m,0}
 +\alpha_2^K(n_e)c_{p-1-m,m+1,0}
 +\alpha_3^K(n_e)c_{p-1-m,m,1}
 \Big].
\end{equation}
Thus the trace uses only the edge row of the B-net, while the normal derivative uses the edge row together with the first adjacent row.  Formula \eqref{eq:normalBcoeff} is the derivative-map counterpart of the smoothness relation \eqref{eq:c1B}.

A single normal is attached to each geometric edge and used on both adjacent triangles.  This gives a common orientation for the derivative trace on both adjacent triangles.  Writing $h_e$ for the edge length, multiplying a normal-derivative row by a nonzero scaling such as $h_e/p$ is allowed for numerical equilibration and does not change the nullspace defining the spline space.

\subsection{Assembly of the macro-patch spline space}\label{subsec:c1matrix}
Concatenate the B-coefficient vectors of all triangles in $\Delta_P$ into
\[
 c_P=(c_{K_1}^{\mathsf T},\ldots,c_{K_{N_T}}^{\mathsf T})^{\mathsf T}.
\]
For every interior edge, insert the B-coefficient smoothness equations \eqref{eq:c0B}--\eqref{eq:c1B}; equivalently, one may insert the trace equations
\begin{equation}\label{eq:c1constraints}
 T_{K^+,e}c_{K^+}-T_{K^-,e}c_{K^-}=0,
 \qquad
 N_{K^+,e}c_{K^+}-N_{K^-,e}c_{K^-}=0,
\end{equation}
with a common orientation of the edge data.  We denote the resulting \emph{smoothness matrix} by $H_P$.  In coefficient form,
\begin{equation}\label{eq:Hkernel}
 S_p^1(\Delta_P)\cong\ker H_P.
\end{equation}
Let $Z_P$ be any numerically stable basis matrix for this kernel.  Then every spline in the local $C^1$ space is represented uniquely in reduced coordinates $a_P$ by
\begin{equation}\label{eq:Sp}
 c_P=Z_Pa_P,
 \qquad
 H_PZ_P=0.
\end{equation}
In the remainder we use the shorthand $S_p^1(P):=S_p^1(\Delta_P)$.

The construction gives exact $C^1$ continuity algebraically at the local-space level.  Every column of $Z_P$ is already the B-coefficient vector of a spline in $S_p^1(P)$.  In particular, the subsequent quasi-Trefftz reduction takes place \emph{inside} the $C^1$ spline space.

For the two-triangle macro-patch used in most experiments, the dimension count is explicit.  The broken degree-$p$ space has
\[
 2N_p=(p+1)(p+2)
\]
B-coefficients.  The common edge contributes $p+1$ value conditions and $p$ first-order smoothness conditions, so
\begin{equation}\label{eq:c1dim2tri}
 \dim S_p^1(P)
 =2N_p-(p+1)-p
 =p^2+p+1.
\end{equation}
For a general multi-triangle macro-patch, edge conditions can interact around vertices and cycles, and the correct dimension is
\begin{equation}\label{eq:c1dimgeneral}
 \dim S_p^1(P)=N_TN_p-\operatorname{rank}H_P.
\end{equation}
Thus the general construction determines the dimension from the assembled smoothness rank.  The same rank-based principle applies to the PDE reduction: if $C_P$ denotes the residual-moment matrix in these conforming coordinates, then
\[
 \dim V_p^{\rm qT}(P)=\dim S_p^1(P)-\operatorname{rank}C_P.
\]
For larger patches there is no universal analogue of the two-triangle formula $2p+1$; vertex compatibility, cycles, patch topology, and residual rank are all captured by the assembled matrices $H_P$ and $C_P$.  The implementation computes both ranks directly from these operators.

\begin{remark}[Why conformity is imposed first]\label{rem:first}
The smoothness matrix $H_P$ depends only on the macro-patch geometry and on the polynomial degree; it is independent of $A$, $\eta$, and $\kappa$.  The first reduction is therefore purely geometric,
\[
 \prod_{K\in\Delta_P}\PP_p(K)
 \longrightarrow S_p^1(P).
\]
On a multi-triangle patch, the image of $S_p^1(P)$ under the differential operator occupies the residual directions compatible with the strongly $C^1$ geometry.  We therefore form $S_p^1(P)$ first and impose the PDE residual equations in these conforming coordinates.  This is the conformity-first principle used throughout the paper.
\end{remark}

\section{\texorpdfstring{$C^1$}{C1} Bernstein quasi-Trefftz spaces}\label{sec:qt}
We consider
\begin{equation}\label{eq:pde}
 \Lop u:=-\nabla\!\cdot(A(x)\nabla u)-\kappa^2\eta(x)u=f
 \qquad\hbox{in }\Omega,
\end{equation}
where $A(x)\in\mathbb R^{2\times2}$ is symmetric and uniformly positive definite and $\eta(x)\ge\eta_0>0$.  Both coefficients may be smooth and nonpolynomial, and the resulting local spaces are quasi-Trefftz.

\subsection{What quasi-Trefftz means in the present method}\label{sec:whatqt}
The defining step of the method is the local equation-adapted reduction. For the homogeneous equation, the exact local Trefftz space would be
\[
 \mathcal T(P)=\{v:\Lop v=0\hbox{ in }P\}.
\]
For variable coefficients this space is infinite dimensional.  We approximate its local equation structure with degree-$p$ $C^1$ Bernstein polynomials that satisfy every residual equation resolvable in degree $p-2$. Set $Y_r(P)=\prod_{K\subset P}\PP_r(K)$ and let $\Pi_r^P$ denote the broken $L^2(P)$ projector onto $Y_r(P)$. The ideal exact-coefficient polynomial quasi-Trefftz space is
\begin{equation}\label{eq:qtprinciple}
 V_p^{\rm qT}(P)=\{v\in S_p^1(P):\Pi_{p-2}^P\Lop v=0\}.
\end{equation}
For the affine problem the corresponding condition is
\begin{equation}\label{eq:qtaffineprinciple}
 \Pi_{p-2}^P(\Lop v-f)=0.
\end{equation}
Thus the low-order interior residual is eliminated \emph{before} patches communicate with one another. The global unknowns are precisely the coefficient directions that survive this local PDE elimination.

This distinguishes the construction from a generic volume polynomial method. A two-triangle degree-$p$ $C^1$ patch contains $p^2+p+1$ conforming polynomial coordinates. When the degree-$(p-2)$ residual-moment map has full row rank, the local equation removes $p(p-1)$ independent residual coordinates, leaving
\[
 (p^2+p+1)-p(p-1)=2p+1
\]
coordinates. Hence the local unknown count has the boundary-scale growth $O(p)$, compared with the $O(p^2)$ conforming volume polynomial count.  The global system therefore receives only the trace-sized reduced coordinates.  At the lowest degree used by the construction, $p=2$, the residual test degree is zero: $Y_{p-2}(P)=Y_0(P)$ contains one constant test function on each triangle.  Thus the two-triangle residual map has two rows, and full row rank gives $7-2=5=2p+1$ reduced coordinates under the present broken-polynomial convention.

\begin{remark}[The polynomial quasi-Trefftz property]\label{rem:notexactpoly}
For $\kappa\ne0$, the highest-degree contribution of the constant-coefficient Helmholtz operator acting on a polynomial comes from $-\kappa^2v$.  The polynomial Trefftz analogue is therefore expressed through the residual moments in \eqref{eq:qtprinciple}: all moments visible at the differential order of the operator are annihilated, and the remaining consistency contribution is confined to higher-order modes.
\end{remark}

Before stating the local approximation result, we define the residual map used in the reduction.

\subsection{Residual moments by integration by parts}  For the quasi-Trefftz reduction we take $r=p-2$.  Let $n_K$ denote the outward unit normal to a triangle $K$.  For $v\in S_p^1(P)$ define the residual moment functional
\begin{equation}\label{eq:moment}
 \langle \mathcal R_P(v),\psi\rangle
 =\sum_{K\subset P}\left[(A\nabla v,\nabla\psi)_K
 -\langle n_K\!\cdot A\nabla v,\psi\rangle_{\partial K}
 -\kappa^2(\eta v,\psi)_K\right],
 \qquad \psi\in Y_r(P).
\end{equation}
This is exactly the moment of the strong residual $\Lop v$, but it requires only point values of $A$ and $\eta$; derivatives of $A$ are never formed.  The boundary contributions are retained triangle by triangle.  For a continuous exact medium and a $C^1$ patch field, contributions from an internal triangle edge cancel because the two normals are opposite and the conormal trace is single-valued; for independently projected trianglewise coefficients, any failure of this cancellation is part of the coefficient-approximation perturbation analyzed below.  In Bernstein coordinates \eqref{eq:moment} is a small dense matrix assembled by quadrature.

Choose a basis of the conforming patch space $S_p^1(P)$, for example the columns of the nullspace matrix $Z_P$ constructed in \Cref{subsec:c1matrix}, and let
\[
 d_P:=\dim S_p^1(P).
\]
The associated synthesis operator
\[
 \mathcal E_P:\mathbb C^{d_P}\longrightarrow S_p^1(P)
\]
maps a vector of conforming patch coordinates to its $C^1$ Bernstein spline.  Composing this synthesis operator with the residual-moment map gives the exact-coefficient local residual matrix
\begin{equation}\label{eq:CPdef}
 C_P:=\mathcal R_P\mathcal E_P.
\end{equation}
The polynomial quasi-Trefftz space is therefore the image under $\mathcal E_P$ of the nullspace of $C_P$.  We write
\[
 \gamma_P:=\sigma_{\min}^{+}(C_P)
\]
for the smallest nonzero singular value of this local residual matrix.  Singular-value and projector statements below use the scaled coefficient coordinates for the conforming space and Riesz-scaled coordinates for the residual test space; changing to another uniformly stable coordinate system only changes the associated continuity constants by the corresponding basis condition numbers.  In these coordinates $\gamma_P^{-1}$ controls sensitivity of the nullspace projector, while the nonzero-singular-value condition number is $\kappa_2^+(C_P)=\sigma_{\max}(C_P)/\gamma_P$.  A small $\gamma_P$ therefore signals a fragile reduction, whereas a moderate ratio $\sigma_{\max}/\gamma_P$ indicates a well-conditioned retained row space.  With these objects defined, the reduction also preserves ordinary polynomial approximation of exact local solutions.

\begin{proposition}[Near-kernel approximation of an exact local solution]\label{prop:qtbest}
Assume $C_P$ has full row rank.  Assume also that the residual-moment map extends continuously to a patch norm $X(P)$ with continuity constant $M_{P,R}$, and that the synthesis operator satisfies $\|\mathcal E_Pa\|_{X(P)}\le M_{P,E}\|a\|_2$. Let $u$ satisfy $\Lop u=0$ on $P$, and let $w=\mathcal E_Pa\in S_p^1(P)$ be any conforming polynomial approximation to $u$. Then there exists $v_P^{\rm qT}\in V_p^{\rm qT}(P)$ such that
\begin{equation}\label{eq:qtbest}
 \|u-v_P^{\rm qT}\|_{X(P)}
 \le \left(1+\frac{M_{P,E}M_{P,R}}{\gamma_P}\right)\|u-w\|_{X(P)}.
\end{equation}
\end{proposition}
\begin{proof}
Set $d=C_P^\dagger C_Pa$ and $v_P^{\rm qT}=\mathcal E_P(a-d)$. Then $C_P(a-d)=0$, hence $v_P^{\rm qT}\in V_p^{\rm qT}(P)$. Because the exact solution has zero residual moments,
\[
 \|d\|_2
 \le \gamma_P^{-1}\|\mathcal R_P(w-u)\|_2
 \le \gamma_P^{-1}M_{P,R}\|w-u\|_{X(P)}.
\]
Therefore
\[
 \|w-v_P^{\rm qT}\|_{X(P)}
 \le M_{P,E}\|d\|_2,
\]
and the triangle inequality gives \eqref{eq:qtbest}.
\end{proof}

Proposition \ref{prop:qtbest} gives the local approximation mechanism: the quasi-Trefftz reduction removes PDE-incompatible polynomial directions from an ordinary $C^1$ approximant. When the scaled singular gap $\gamma_P$ remains separated from zero in the resolved regime, the reduced space inherits the approximation order of the underlying Bernstein spline space up to the local stability factor in \eqref{eq:qtbest}.

\subsection{Adaptive algebraic-degree projection of nonpolynomial coefficients}
For fast patch construction we project the medium locally, leaving the PDE solution unprojected.  On each triangle choose the smallest even degree $m_K$ for which the Bernstein $L^2$ projections $A_{m_K}$ and $\eta_{m_K}$ satisfy
\begin{equation}\label{eq:coefftol}
 \frac{\left(\norm{A-A_{m_K}}_{L^2(K)}^2+\norm{\eta-\eta_{m_K}}_{L^2(K)}^2\right)^{1/2}}
 {\left(\norm{A}_{L^2(K)}^2+\norm{\eta}_{L^2(K)}^2\right)^{1/2}}
 \le\varepsilon_A.
\end{equation}
For the matrix field $A$, the $L^2$ norm in \eqref{eq:coefftol} is the componentwise Frobenius norm.  We call this adaptive algebraic degree (AAD).  The local quasi-Trefftz basis uses $A_{m_K},\eta_{m_K}$, and the global graph residual in \Cref{sec:graph} is assembled with the original $A,\eta,f$.  Thus the coefficient projection is a basis-construction device.  In the experiments $\varepsilon_A=10^{-9}$; tightening it changes the reported errors below the displayed digits.

AAD uses a single precomputed projection hierarchy.  One high-order quadrature cloud is fixed on the reference triangle and the weighted Bernstein projection maps for all candidate degrees are precomputed.  On a physical triangle $A$ and $\eta$ are evaluated once on the mapped cloud.  Every candidate degree is then obtained by a matrix product, and the smallest degree satisfying \eqref{eq:coefftol} is selected.  The implementation searches candidate degrees in steps of two; this halves the number of projection tests and is only an implementation choice, since odd algebraic degrees are mathematically admissible as well.  Thus every candidate degree reuses the same coefficient samples and precomputed mass-matrix factors.

\subsection{Coefficient approximation and residual-map perturbation}\label{sec:coefferr}
The coefficient projection is local, so its effect can be quantified before any global solve.  Write
\[
 \delta A_K=A-A_{m_K},\qquad \delta\eta_K=\eta-\eta_{m_K},
\]
and, when a projected source is used, $\delta f_K=f-f_{m_K}$.  The implementation used for the experiments integrates the true $f$, hence $\delta f_K=0$ there; retaining $\delta f_K$ below makes the estimate applicable to a fully algebraic source projection as well.

For $v\in S_p^1(P)$ and $\psi\in Y_r(P)$ define the local coefficient perturbation functional
\begin{align}
 \mathfrak E_P(v,\psi):=
 \sum_{K\subset P}\{&\|\delta A_K\|_{L^\infty(K)}\|\nabla v\|_{L^2(K)}\|\nabla\psi\|_{L^2(K)} \\[-1mm]
 &+\|\delta A_K\|_{L^\infty(\partial K)}\|\nabla v\|_{L^2(\partial K)}\|\psi\|_{L^2(\partial K)}
 +\kappa^2\|\delta\eta_K\|_{L^\infty(K)}\|v\|_{L^2(K)}\|\psi\|_{L^2(K)}\}.
 \label{eq:coefE}
\end{align}
The boundary contribution is essential: the moment identity \eqref{eq:moment} contains the conormal trace, so an $L^2(K)$ coefficient indicator alone serves as an acceptance criterion for AAD, while a complete perturbation norm also requires the boundary term.  For smooth coefficients the corresponding edge defect follows from the same polynomial approximation by a trace estimate; in code it can also be audited directly on the edge quadrature cloud.

\begin{proposition}[Perturbation of residual moments]\label{prop:coeffmoment}
Let $\mathcal R_P$ and $\widetilde{\mathcal R}_P$ denote the residual-moment maps formed with $(A,\eta)$ and $(A_m,\eta_m)$, respectively.  Then
\begin{equation}\label{eq:momentpert}
 |\langle(\mathcal R_P-\widetilde{\mathcal R}_P)v,\psi\rangle|
 \le \mathfrak E_P(v,\psi).
\end{equation}
If the source is also projected, its moment perturbation satisfies
\begin{equation}\label{eq:sourcepert}
 |(f-f_m,\psi)_P|\le \sum_{K\subset P}\|\delta f_K\|_{L^2(K)}\|\psi\|_{L^2(K)}.
\end{equation}
\end{proposition}
\begin{proof}
Subtract the two versions of \eqref{eq:moment}.  The three resulting terms are the volume diffusion term, the conormal boundary term, and the reaction term.  Apply Cauchy--Schwarz to each and use the $L^\infty$ coefficient bounds.  Equation \eqref{eq:sourcepert} is immediate.
\end{proof}

For fixed polynomial degrees, inverse and trace inequalities turn \eqref{eq:coefE} into a purely coefficient-based quantity.  For example, on a shape-regular triangle,
\begin{align}\label{eq:coefscale}
 \mathfrak E_P(v,\psi)
 &\le C_{\rm inv}(p,r)\sum_{K\subset P}
 \left[
 h_K^{-1}\|\delta A_K\|_{L^\infty(K)}
 +h_K^{-1}\|\delta A_K\|_{L^\infty(\partial K)}
 +\kappa^2\|\delta\eta_K\|_{L^\infty(K)}
 \right] \notag\\
 &\qquad\times \|v\|_{H^1(K)}\|\psi\|_{L^2(K)}.
\end{align}
where $C_{\rm inv}(p,r)$ is explicit on the reference triangle and is precomputable for every $(p,r)$ used by the code.  We retain the trace-explicit form \eqref{eq:coefE} in the analysis so that the polynomial-degree factors remain visible.

The AAD tolerance has the expected approximation meaning.  If a scalar component $a$ of $A$, $\eta$, or $f$ belongs to $H^s(K)$, $1\le s\le m+1$, and the local polynomial projector is uniformly stable in the required norms, standard polynomial approximation and the trace theorem give
\begin{align}
 \|a-\Pi_m a\|_{L^2(K)}+\frac{h_K}{m+1}\|\nabla(a-\Pi_m a)\|_{L^2(K)}
 &\le C\left(\frac{h_K}{m+1}\right)^s |a|_{H^s(K)},\label{eq:coeffapprox}\\
 \|a-\Pi_m a\|_{L^2(\partial K)}
 &\le C\left(\frac{h_K}{m+1}\right)^{s-1/2} |a|_{H^s(K)}.\label{eq:coefftrace}
\end{align}
For analytic coefficient fields the observed degree convergence is substantially faster; this is why the nonpolynomial matrix experiments below typically stop at AAD degrees between six and eight.

To connect coefficient approximation to the reduced space, let $C_P$ be the exact-coefficient homogeneous residual matrix on $S_p^1(P)$ and $\widetilde C_P=C_P+\Delta C_P$ its AAD counterpart, both expressed in the same scaled coefficient and residual Riesz coordinates.  Let
\[
 \gamma_P=\sigma_{\min}^{+}(C_P)
\]
be the smallest nonzero singular value of the exact local residual map.

\begin{proposition}[Rank and kernel stability]\label{prop:kernelpert}
Assume $C_P$ has full row rank. If $\|\Delta C_P\|_2<\gamma_P$, then $C_P$ and $\widetilde C_P$ have the same row rank.  If $\Pi_P$ and $\widetilde\Pi_P$ are the orthogonal projectors onto their kernels, then
\begin{equation}\label{eq:gap}
 \|\Pi_P-\widetilde\Pi_P\|_2
 \le
 \frac{\|\Delta C_P\|_2}{\gamma_P-\|\Delta C_P\|_2}.
\end{equation}
In particular, for $\|\Delta C_P\|_2\le\gamma_P/2$ the local quasi-Trefftz space rotates by at most $2\|\Delta C_P\|_2/\gamma_P$.
\end{proposition}
\begin{proof}
Weyl's singular-value perturbation bound preserves the positive singular-value gap.  Applying the standard sin-$\Theta$ perturbation estimate to the right singular subspaces associated with the zero singular values gives \eqref{eq:gap}.
\end{proof}

The quantity $\gamma_P$ is therefore the correct scale against which the coefficient projection error should be judged.  This gives a sharper interpretation of the numerical rank test used in the implementation: AAD is harmless when the induced residual-map perturbation is small relative to the retained singular gap.

\begin{proposition}[Projected affine consistency]\label{prop:affinepert}
Suppose the AAD affine coordinate $\widetilde a_P$ satisfies
\[
 \widetilde C_P\widetilde a_P=\widetilde b_P,
\]
where $b_P$ and $\widetilde b_P$ are the true and projected source moment vectors.  Then its residual with respect to the exact local moment equations obeys
\begin{equation}\label{eq:affinecons}
 \|C_P\widetilde a_P-b_P\|_2
 \le \|\Delta C_P\|_2\,\|\widetilde a_P\|_2
      +\|\widetilde b_P-b_P\|_2.
\end{equation}
Thus a true-source implementation has only the first term on the right-hand side.
\end{proposition}
\begin{proof}
Use $C_P=\widetilde C_P-\Delta C_P$ and insert the projected equation.
\end{proof}

\subsection{Affine quasi-Trefftz patch space for nonzero forcing}\label{sec:affine}
The $C^1$ coordinates are the coefficients $a_P$ in \eqref{eq:Sp}, so that the full B-coefficient vector is $c_P=Z_Pa_P$.  Let $\widetilde C_P$ denote the degree-$(p-2)$ residual-moment matrix in these $C^1$ coordinates, assembled with the accepted AAD approximation of the medium, and let $\widetilde b_P$ contain the corresponding source moments.  Choose a particular coordinate vector $a_P^f$ and a homogeneous coordinate matrix $W_P$ satisfying
\begin{equation}\label{eq:affine}
 \widetilde C_Pa_P^f=\widetilde b_P,
 \qquad
 \widetilde C_PW_P=0.
\end{equation}
Any particular solution of the residual-moment equations generates the same affine trial set, because two such solutions differ by an element of the homogeneous local space.  In full Bernstein--B\'ezier coordinates the particular lift and the homogeneous quasi-Trefftz basis are
\[
 c_P^f=Z_Pa_P^f,
 \qquad
 Q_P=Z_PW_P.
\]
Thus
\begin{equation}\label{eq:localQT}
 u_h|_P=u_P^f+\sum_{j=1}^{r_P}(z_P)_j\,q_{P,j},
 \qquad
 \operatorname{coeff}_B(u_P^f)=c_P^f,
 \qquad
 \operatorname{coeff}_B(q_{P,j})=(Q_P)_{:j},
\end{equation}
where $r_P=\dim\ker\widetilde C_P$.  The source changes only the affine lift; the globally coupled unknowns are the homogeneous coordinates $z_P$.

\begin{proposition}[Trace-sized local dimension]\label{prop:dim}
Assume the degree-$(p-2)$ residual-moment map in \eqref{eq:affine} has full row rank on a two-triangle $C^1$ macro-patch.  Then
\[
 \dim\ker\widetilde C_P=2p+1.
\]
The same conclusion is stable under sufficiently small coefficient-projection perturbations.
\end{proposition}
\begin{proof}
The broken degree-$p$ space on two triangles has dimension $(p+1)(p+2)$.  Strong $C^1$ matching across the common edge imposes $p+1$ value conditions and $p$ normal-derivative conditions, so $\dim S_p^1(P)=p^2+p+1$.  The broken degree-$(p-2)$ residual space has dimension $p(p-1)$.  Rank--nullity therefore gives $2p+1$.  Rank is locally constant under perturbations smaller than the smallest retained singular value.
\end{proof}

\begin{remark}
The implementation uses the computed rank directly, with $2p+1$ emerging on the full-row-rank two-triangle patches.  The implementation records the numerical rank and residual of every local solve.  In all heterogeneous tests below every patch retained exactly $2p+1$ homogeneous coordinates, and the projected forcing moments were matched to roundoff.
\end{remark}

\begin{proposition}[Coefficient-perturbed local approximation]\label{prop:approx}
Let $V_P=\ker C_P$ be the exact-coefficient homogeneous local space and $\widetilde V_P=\ker\widetilde C_P$ the AAD space.  Assume $\|\Delta C_P\|_2<\gamma_P$ and let
\[
 \delta_P:=\frac{\|\Delta C_P\|_2}{\gamma_P-\|\Delta C_P\|_2}.
\]
For every coefficient vector $a\in V_P$ there exists $\widetilde a\in\widetilde V_P$ such that
\begin{equation}\label{eq:localapproxpert}
 \|a-\widetilde a\|_2\le \delta_P\|a\|_2.
\end{equation}
Consequently, if $\mathcal E_P$ denotes the Bernstein synthesis map into any norm $X(P)$ and $M_{P,X}=\|\mathcal E_P\|_{2\to X(P)}$, then
\begin{equation}\label{eq:localapproxphys}
 \inf_{\widetilde v\in\widetilde V_P}\|v-\widetilde v\|_{X(P)}
 \le M_{P,X}\delta_P\|a\|_2,
 \qquad v=\mathcal E_Pa\in V_P.
\end{equation}
Thus the coefficient approximation contributes through the dimensionless ratio $\|\Delta C_P\|_2/\gamma_P$.
\end{proposition}
\begin{proof}
Take $\widetilde a=\widetilde\Pi_Pa$.  Since $a=\Pi_Pa$, \eqref{eq:localapproxpert} follows from \Cref{prop:kernelpert}.  Applying the synthesis map gives \eqref{eq:localapproxphys}.
\end{proof}

\subsection{The surviving quasi-Trefftz consistency defect}\label{sec:qtdefect}
The degree-$(p-2)$ equations define the reduced coordinates and annihilate the resolvable low-order residual.  The remaining volume quantity is precisely the higher-order PDE defect that survives the quasi-Trefftz reduction.

Let $\Lop^{(m)}$ denote the operator formed with the AAD coefficients and let $f_m$ denote the source used in the local affine constraint. For every implemented affine patch function $v\in u_P^f+\widetilde V_p^{\rm qT}(P)$,
\begin{equation}\label{eq:lowannihilate}
 \Pi_{p-2}^P(\Lop^{(m)} v-f_m)=0.
\end{equation}
Consequently the oversampled true residual has the exact decomposition
\begin{align}
 \Pi_p^P(\Lop v-f)
 ={}&\Pi_p^P\big[(\Lop-\Lop^{(m)})v-(f-f_m)\big] \label{eq:defectsplit}\\
 &+(I-\Pi_{p-2}^P)\Pi_p^P(\Lop^{(m)} v-f_m).\notag
\end{align}
The first term is purely the coefficient/source approximation defect analyzed in \Cref{sec:coefferr}. The second is the \emph{Trefftz truncation defect}: by \eqref{eq:lowannihilate} it contains only residual modes above degree $p-2$. This identity is central to the interpretation of the method.

If the projected medium and source are exact and the polynomial space happens to contain a local exact solution, both terms in \eqref{eq:defectsplit} vanish and the method reduces to a trace-only Trefftz coupling. More generally, the second term is the high-order residual component controlled by polynomial approximation as degree or patch resolution increase, while the first is driven below the discretization scale by AAD. The volume channel therefore measures the departure from the exact Trefftz property.

For global coupling we retain the complementary degree-$p$ residual directions that survive the local construction.  Let $\rho_{P,p}^{\perp}(v;f)$ denote the $L^2(P)$ Riesz representative of
\[
 (I-\Pi_{p-2}^P)\Pi_p^P(\Lop v-f).
\]
For the exact solution this quantity vanishes.  For an AAD patch function, \eqref{eq:defectsplit} shows that it consists of the surviving quasi-Trefftz residual together with the corresponding high-order part of the coefficient-approximation defect.  The omitted degree-$(p-2)$ true-medium component is treated explicitly in the physical-energy estimate below.  In particular,
\begin{equation}\label{eq:defectbound}
 \|\rho_{P,p}^{\perp}(v;f)\|_{L^2(P)}
 \le \|\Pi_p^P[(\Lop-\Lop^{(m)})v-(f-f_m)]\|_{L^2(P)}
 +\|(I-\Pi_{p-2}^P)\Pi_p^P(\Lop^{(m)} v-f_m)\|_{L^2(P)}.
\end{equation}
Thus the cell term used below measures precisely the two controlled departures from an exact Trefftz patch: medium approximation and finite polynomial resolution.

\section{Intrinsic graph-residual coupling}\label{sec:graph}
Let $\Patches_h$ be a nonoverlapping partition of $\Omega$ into macro-patches, let $\F_h^i$ denote the set of interfaces between distinct macro-patches, and set $\Gamma=\partial\Omega$.  On each patch let $\widetilde V_p^{\rm qT}(P)$ be the implemented homogeneous AAD space and let $u_P^f$ be the local particular lift from \Cref{sec:affine}.  The homogeneous broken space and the affine trial set are
\begin{equation}\label{eq:broken}
 \widetilde V_h^{\rm qT}:=\prod_{P\in\Patches_h}\widetilde V_p^{\rm qT}(P),
 \qquad
 \mathcal V_h^f:=\prod_{P\in\Patches_h}\bigl(u_P^f+\widetilde V_p^{\rm qT}(P)\bigr).
\end{equation}
Every function in $\mathcal V_h^f$ is strongly $C^1$ inside each macro-patch, and distinct patches are coupled through the global Cauchy graph.
For an interface $e=P^+\cap P^-$ fix one normal $n_e$ and set
\[
 \jump v=v^+-v^-,
 \qquad
 \jump{q_n(v)}=n_e\!\cdot A\nabla v^+-n_e\!\cdot A\nabla v^-.
\]
This fixed-normal convention avoids sign ambiguity in the Cauchy trace.  On an interior interface define the characteristic scale
\[
 Z_e=\kappa\sqrt{\eta\,n_e^T A n_e}.
\]

We solve the interior impedance problem
\begin{equation}\label{eq:iip}
 \Lop u=f\quad\hbox{in }\Omega,
 \qquad
 n\!\cdot A\nabla u-iZ_nu=g\quad\hbox{on }\Gamma,\qquad Z_n=\kappa\sqrt{\eta\,n^TA n}.
\end{equation}
For $v\in \mathcal V_h^f$ define the graph residual
\begin{align}
 \mathcal J_h(v)=
 &\sum_{P\in\Patches_h}\kappa^{-2}\norm{\rho_{P,p}^{\perp}(v;f)}_{L^2(P)}^2 \label{eq:J}\\
 &+\sum_{e\in\F_h^i}
 \left(
 \norm{Z_e^{1/2}\jump v}_{L^2(e)}^2
 +\norm{Z_e^{-1/2}\jump{q_n(v)}}_{L^2(e)}^2
 \right)\notag\\
 &+\norm{Z_n^{-1/2}\bigl(n\!\cdot A\nabla v-iZ_nv-g\bigr)}_{L^2(\Gamma)}^2,\notag
\end{align}
The discrete solution is
\begin{equation}\label{eq:min}
 u_h=\arg\min_{v_h\in \mathcal V_h^f}\mathcal J_h(v_h).
\end{equation}

The scale $Z_e$ is the natural local characteristic impedance.  For $A=I$ and $\eta=1$ it reduces to $\kappa$, recovering the constant-coefficient Cauchy graph.  The jump terms are the two components of the scaled value/conormal mismatch.

\begin{remark}[Why a quasi-Trefftz method still has a cell term]\label{rem:cellterm}
The first line of \eqref{eq:J} measures the surviving consistency defect after the local reduction.  The local equations \eqref{eq:lowannihilate} have already eliminated all degree-$(p-2)$ residual directions; on the full-row-rank two-triangle patches considered numerically, this reduces the local coordinates from $O(p^2)$ to $2p+1$.  The cell term acts entirely on these reduced coordinates.  For an exact Trefftz basis it vanishes identically and \eqref{eq:J} reduces to Cauchy-trace and physical-boundary matching.

The retained complement has a precise role: it measures the polynomial residual directions that survive the local near-kernel reduction.  The volume channel therefore records the consistency cost associated with the polynomial near-kernel representation.
\end{remark}

\subsection{Residual optimality and discrete injectivity}
Let $\mathcal B_h$ denote the complete linear residual map appearing in \eqref{eq:J}, with homogeneous data.  It acts on differences of broken functions; in particular, the difference of two members of $\mathcal V_h^f$ lies in $\widetilde V_h^{\rm qT}$.  The graph seminorm is
\[
 \norm{w}_{G,h}:=\norm{\mathcal B_hw}_2.
\]

\begin{theorem}[Finite-dimensional well posedness]\label{thm:norm}
If $\mathcal B_h$ is injective on $\widetilde V_h^{\rm qT}$, then $\norm{\cdot}_{G,h}$ is a norm on the homogeneous discrete space, \eqref{eq:min} has a unique solution in the affine set $\mathcal V_h^f$, and the normal-equation matrix is Hermitian positive definite.
\end{theorem}
\begin{proof}
Injectivity gives positive definiteness of $\mathcal B_h^*\mathcal B_h$ on the reduced coordinates.  Existence and uniqueness then follow from finite-dimensional least squares.
\end{proof}

\begin{theorem}[Graph-norm best approximation]\label{thm:quasi}
Let $u$ solve the heterogeneous impedance problem and define $\mathcal B_hu$ using the same true-coefficient residual moments and traces.  Then
\begin{equation}\label{eq:quasi}
 \norm{u-u_h}_{G,h}=\min_{v_h\in \mathcal V_h^f}\norm{u-v_h}_{G,h}.
\end{equation}
\end{theorem}
\begin{proof}
The exact solution satisfies all volume moments, interface transmission conditions, and the physical boundary condition.  Hence the least-squares functional is exactly the squared graph distance from $u$.
\end{proof}

\subsection{Physical-energy reliability at high frequency}\label{sec:physical}
Graph-residual best approximation is exact, and the physical error estimate follows from a stability transfer from residuals and Cauchy mismatches to the broken energy norm.  The optimized formulation gives a more explicit transfer than a generic discrete norming assumption because the locally annihilated degree-$(p-2)$ residual modes are removed before global assembly.

For a broken function $w$ set
\begin{equation}\label{eq:energynorm}
 \|w\|_{1,\kappa,A,\eta}^2
 :=\sum_{P\in\Patches_h}
 \left(\|A^{1/2}\nabla w\|_{L^2(P)}^2
 +\kappa^2\|\eta^{1/2}w\|_{L^2(P)}^2\right),
\end{equation}
and define the broken negative norm by
\[
 \|r\|_{-1,h}^2:=\sum_{P\in\Patches_h}\|r\|_{H^{-1}(P)}^2.
\]
Let $\Pi_{p-2}^P$ be the $L^2(P)$ projector onto the broken polynomial space of degree at most $p-2$ on the triangles of $P$.

\begin{lemma}[Negative norm of the surviving polynomial residual]\label{lem:negpoly}
Assume the macro-patches are uniformly shape regular.  If $q_P\in Y_p(P)$ satisfies
\[
 (q_P,r)_{P}=0\qquad\text{for every }r\in Y_{p-2}(P),
\]
then
\begin{equation}\label{eq:negpoly}
 \|q_P\|_{H^{-1}(P)}
 \le C_{\rm app}\frac{h_P}{p}\,\|q_P\|_{L^2(P)},
\end{equation}
where $C_{\rm app}$ depends only on the shape-regularity class of the patch.  For families of larger multi-triangle patches, the same statement requires uniform control of the patch cardinality and Poincar\'e/chunkiness constants; the two-triangle families used below satisfy this automatically.
\end{lemma}
\begin{proof}
For $\phi\in H_0^1(P)$, orthogonality gives
\[
 (q_P,\phi)_P=(q_P,\phi-\Pi_{p-2}^P\phi)_P.
\]
The standard $hp$ polynomial projection estimate on a shape-regular triangle, applied elementwise on the patch (see, e.g., \cite[Chap.~3]{Schwab1998}), yields
\[
 \|\phi-\Pi_{p-2}^P\phi\|_{L^2(P)}
 \le C_{\rm app}\frac{h_P}{p}\|\nabla\phi\|_{L^2(P)}.
\]
Taking the supremum over $\phi$ proves \eqref{eq:negpoly}.
\end{proof}

For any broken $w$ define the unresolved residual tail
\begin{equation}\label{eq:restail}
 \mathcal R_{>p}(w)
 :=\left(\sum_{P\in\Patches_h}
 \|(I-\Pi_p^P)\Lop w\|_{H^{-1}(P)}^2\right)^{1/2}.
\end{equation}
For the error $e=u-u_h$ one may equivalently replace $\Lop e$ by $f-\Lop u_h$.  Thus $\mathcal R_{>p}(e)$ measures the physical residual tail above the degree-$p$ volume channel of the graph functional.

For the implemented AAD space there is also a low-order coefficient defect
\begin{equation}\label{eq:lowdef}
 \mathcal D_{\rm AAD}(w)
 :=\left(\sum_{P\in\Patches_h}
 \|\Pi_{p-2}^P\Lop w\|_{H^{-1}(P)}^2\right)^{1/2}.
\end{equation}
For an exact-coefficient quasi-Trefftz error this term vanishes.  For the implemented error $e=u-u_h$, using \eqref{eq:lowannihilate} gives
\begin{equation}\label{eq:lowdefsplit}
 \Pi_{p-2}^P\Lop e
 =-\Pi_{p-2}^P\big[(\Lop-\Lop^{(m)})u_h-(f-f_m)\big],
\end{equation}
so $\mathcal D_{\rm AAD}(e)$ is controlled directly by the accepted coefficient/source approximation, with wave resolution governed separately by $\kappa h/p$.

\begin{lemma}[High-order residual tail]\label{lem:restail}
Assume the macro-patches are uniformly shape regular.  For every $r\in L^2(P)$,
\begin{equation}\label{eq:restailhp}
 \|(I-\Pi_p^P)r\|_{H^{-1}(P)}
 \le C_{\rm app}\frac{h_P}{p+1}\,\|(I-\Pi_p^P)r\|_{L^2(P)}.
\end{equation}
If $r$ is piecewise $H^s$ on the triangles of $P$, with $s>0$, then
\begin{equation}\label{eq:restailsmooth}
 \|(I-\Pi_p^P)r\|_{H^{-1}(P)}
 \le C_s\left(\frac{h_P}{p+1}\right)^{s+1}|r|_{H^s(\Delta_P)}.
\end{equation}
\end{lemma}
\begin{proof}
The first estimate follows exactly as in \Cref{lem:negpoly}, now using orthogonality to $Y_p(P)$.  The second follows by combining \eqref{eq:restailhp} with the standard elementwise $hp$ projection estimate for degree $p$ \cite[Chap.~3]{Schwab1998}.
\end{proof}

The low-order AAD term can also be bounded by quantities already controlled in local basis construction.  Write $A_m$ and $\eta_m$ for the accepted local coefficient approximants, and set
\[
 \varepsilon_{A,0,P}=\|A-A_m\|_{L^\infty(P)},\qquad
 \varepsilon_{A,1,P}=\|\operatorname{div}(A-A_m)\|_{L^\infty(P)},\qquad
 \varepsilon_{\eta,P}=\|\eta-\eta_m\|_{L^\infty(P)}.
\]
Here the divergence of a matrix is taken columnwise, consistently with $\operatorname{div}(A\nabla v)$.

\begin{proposition}[Coefficient control of the omitted low-order residual]\label{prop:aadphysical}
Assume $A\in W^{1,\infty}(P)^{2\times2}$ and let $u_h$ be polynomial on each triangle.  Then the low-order defect of the discrete error satisfies
\begin{equation}\label{eq:aadphysical}
\begin{aligned}
 \mathcal D_{\rm AAD}(u-u_h)
 \le C\Bigg(\sum_{P\in\Patches_h} h_P^2
 \Big[&\varepsilon_{A,0,P}\|D^2u_h\|_{L^2(P)}
 +\varepsilon_{A,1,P}\|\nabla u_h\|_{L^2(P)}\\
 &+\kappa^2\varepsilon_{\eta,P}\|u_h\|_{L^2(P)}
 +\|f-f_m\|_{L^2(P)}\Big]^2\Bigg)^{1/2}.
\end{aligned}
\end{equation}
with a constant depending only on the patch shape-regularity class.  In particular this term vanishes for an exactly represented polynomial medium and source.
\end{proposition}
\begin{proof}
From \eqref{eq:lowdefsplit}, $L^2$ contractivity of $\Pi_{p-2}^P$, and the local Poincar\'e inequality,
\[
 \|\Pi_{p-2}^P\Lop(u-u_h)\|_{H^{-1}(P)}
 \le Ch_P\| (\Lop-\Lop^{(m)})u_h-(f-f_m)\|_{L^2(P)}.
\]
Expanding $-\operatorname{div}((A-A_m)\nabla u_h)$ and applying the $L^\infty$ coefficient bounds gives \eqref{eq:aadphysical}; summation over patches finishes the proof.
\end{proof}

Thus both terms absent from the optimized graph functional have explicit origins: \Cref{prop:aadphysical} ties the low-order component to coefficient/source approximation, while \Cref{lem:restail} gives an additional $h_P/(p+1)$ gain for the unresolved high-order residual.

Let $\mathcal S_h(w)$ denote the square root of the two Cauchy-interface terms and the physical-boundary term in \eqref{eq:J}.  Continuous resolvent estimates for variable-coefficient Helmholtz problems are available under geometric and coefficient assumptions; see, for example, \cite{GrahamPemberySpence2019,GrahamSauter2020}.  To separate that continuous issue from the discrete reduction, we assume the following impedance stability and lifting estimate: after lifting the interface Cauchy jumps of $w$ to a globally admissible impedance error,
\begin{equation}\label{eq:contstab}
 \|w\|_{1,\kappa,A,\eta}
 \le C_{\rm imp}(\kappa,A,\eta)
 \left(\|\Lop w\|_{-1,h}+C_{\rm lift}\mathcal S_h(w)\right),
\end{equation}
where $C_{\rm lift}$ is bounded on the shape-regular patch family.  This assumption concerns the continuous impedance stability and the broken Cauchy lifting.  The analysis does not require a wavenumber-uniform bound for $C_{\rm imp}$ in a general heterogeneous anisotropic medium: its dependence on $\kappa$, the domain, and the coefficient geometry belongs to the continuous resolvent problem.  Likewise, $C_{\rm lift}$ is a geometry-dependent lifting constant controlled on the fixed shape-regular patch family.  The pollution study below probes the discrete resolution factor $\kappa h/p$ in \eqref{eq:physicalrel}; it does not estimate these continuous constants.  The discrete quasi-Trefftz contribution is handled by the explicit residual decomposition above.

\begin{proposition}[Physical-energy reliability]\label{prop:reliable}
Assume \eqref{eq:contstab}.  Let
\[
 \chi_h:=\max_{P\in\Patches_h}\frac{\kappa h_P}{p}.
\]
Then every error $w=u-v_h$ with $v_h\in\mathcal V_h^f$ satisfies
\begin{equation}\label{eq:physicalrel}
 \|w\|_{1,\kappa,A,\eta}
 \le C_{\rm imp}
 \left[
 \bigl(C_{\rm lift}+C_{\rm app}\chi_h\bigr)\|w\|_{G,h}
 +\mathcal D_{\rm AAD}(w)+\mathcal R_{>p}(w)
 \right],
\end{equation}
where $\mathcal D_{\rm AAD}$ is absent for an exact-coefficient quasi-Trefftz space.  Consequently the discrete solution satisfies
\begin{equation}\label{eq:physicalerror}
 \|u-u_h\|_{1,\kappa,A,\eta}
 \le C_{\rm imp}
 \left[
 \bigl(C_{\rm lift}+C_{\rm app}\chi_h\bigr)
 \inf_{v_h\in\mathcal V_h^f}\|u-v_h\|_{G,h}
 +\mathcal D_{\rm AAD}(u-u_h)+\mathcal R_{>p}(u-u_h)
 \right].
\end{equation}
For exact polynomial media whose physical residual is represented in the degree-$p$ channel, both additional terms vanish and \eqref{eq:physicalerror} is an explicit graph-to-energy transfer up to the continuous impedance and lifting constants.  For smooth nonpolynomial media, \Cref{lem:restail,prop:aadphysical} provide explicit bounds for the two additional terms.
\end{proposition}
\begin{proof}
On each patch decompose the physical residual into three mutually separated pieces,
\[
 \Lop w
 =\Pi_{p-2}^P\Lop w
 +(I-\Pi_{p-2}^P)\Pi_p^P\Lop w
 +(I-\Pi_p^P)\Lop w.
\]
The middle term lies in the $L^2$-orthogonal complement of $Y_{p-2}(P)$, so \Cref{lem:negpoly} gives
\[
 \|(I-\Pi_{p-2}^P)\Pi_p^P\Lop w\|_{H^{-1}(P)}
 \le C_{\rm app}\frac{h_P}{p}
 \|(I-\Pi_{p-2}^P)\Pi_p^P\Lop w\|_{L^2(P)}.
\]
Because the graph residual scales its represented volume channel by $\kappa^{-1}$, summing over patches bounds this contribution by $C_{\rm app}\chi_h\|w\|_{G,h}$.  The first and third pieces are exactly the terms \eqref{eq:lowdef} and \eqref{eq:restail}.  The Cauchy lifting is bounded by $C_{\rm lift}\mathcal S_h(w)\le C_{\rm lift}\|w\|_{G,h}$.  Substitution in \eqref{eq:contstab} proves \eqref{eq:physicalrel}; \eqref{eq:physicalerror} follows from \Cref{thm:quasi}.
\end{proof}

\begin{corollary}[Exact constant-coefficient case]\label{cor:constantphysical}
Assume \eqref{eq:contstab}, and suppose that $A$ and $\eta$ are constant on each macro-patch and that the source is represented exactly in $Y_p(P)$.  Then the low-order coefficient defect and the residual tail vanish, and
\begin{equation}\label{eq:constantphysical}
 \|u-u_h\|_{1,\kappa,A,\eta}
 \le C_{\rm imp}
 \bigl(C_{\rm lift}+C_{\rm app}\chi_h\bigr)
 \inf_{v_h\in\mathcal V_h^f}\|u-v_h\|_{G,h}.
\end{equation}
In particular, the discrete graph-to-energy amplification is explicit in $\chi_h=\max_P \kappa h_P/p$; any additional frequency dependence is confined to the continuous impedance and lifting constants.
\end{corollary}
\begin{proof}
For constant $A$ and $\eta$, the residual of a degree-$p$ patch polynomial has degree at most $p$.  Exact coefficient and source representation gives $\mathcal D_{\rm AAD}=0$, while $\mathcal R_{>p}=0$.  Equation \eqref{eq:constantphysical} follows from \Cref{prop:reliable}.
\end{proof}

\begin{remark}[Connection with the pollution experiment]\label{rem:physicalpollution}
The factor $\chi_h=\max_P \kappa h_P/p$ is, up to the uniform equivalence between patch and triangle diameters on the paired mesh family, the resolution parameter varied in the pollution study below.  The estimate is a reliability statement whose constants $C_{\rm imp}$, the lifting constant, and the smoothness seminorm in \eqref{eq:restailsmooth} can carry frequency dependence.  The estimate identifies the discrete mechanism explicitly.  Keeping $\kappa h/p$ bounded prevents the represented high-order residual from acquiring an additional resolution-dependent amplification when graph control is transferred to the physical energy norm.
\end{remark}

The role of AAD can now be separated from discretization error.  Let $\mathcal V_h^{f,{\rm ex}}$ denote the hypothetical affine trial set formed with exact coefficients and exact source moments, and let $\mathcal V_h^f$ be the implemented AAD trial set in \eqref{eq:broken}.  Suppose the graph residual of a patch polynomial is continuous with respect to its scaled Bernstein coefficient vector, with constant $M_{P,G}$.  Combining \Cref{prop:approx,thm:quasi} gives the Strang-type decomposition
\begin{align}
 \|u-u_h\|_{G,h}
 &\le \inf_{v_h\in \mathcal V_h^{f,{\rm ex}}}\|u-v_h\|_{G,h}
 +\left(\sum_{P\in\Patches_h}M_{P,G}^2\delta_P^2\|a_P(v_h^*)\|_2^2\right)^{1/2}
 +\mathcal E_f,\label{eq:strangcoeff}
\end{align}
where $v_h^*\in\mathcal V_h^{f,{\rm ex}}$ is any exact-coefficient comparator and $\mathcal E_f$ is zero when the source moments use the true $f$; otherwise it is controlled by \eqref{eq:sourcepert}.  Hence coefficient approximation preserves the graph best-approximation mechanism and contributes a separately controllable term whose natural local indicator is
\begin{equation}\label{eq:aadindicator}
 \boxed{\quad \delta_P\simeq
 \frac{\|C_P-\widetilde C_P\|_2}{\sigma_{\min}^{+}(C_P)}.\quad}
\end{equation}
In computations $C_P$ is unavailable because it would require the exact nonpolynomial operator, but the numerator can be estimated by successive AAD degrees and the denominator by the retained singular gap of the accepted matrix.  This suggests the practical stopping test
\begin{equation}\label{eq:aadstop}
 \frac{\|\widetilde C_P^{(m+2)}-\widetilde C_P^{(m)}\|_2}
 {\sigma_{\min}^{+}(\widetilde C_P^{(m+2)})}
 \le \tau_{\rm AAD},
\end{equation}
in addition to the coefficient-space test \eqref{eq:coefftol}.  The experiments below use the coefficient-space test; \eqref{eq:aadstop} is the more directly operator-relevant criterion for future adaptive implementations.

\section{Algebra and implementation}\label{sec:impl}
The implementation mirrors the two local reductions and then stays in reduced coordinates.  This section records the algebra in more detail.  The Bernstein--B\'ezier continuation formula makes the conformity stage explicit, and all geometry enters through affine maps and local coefficient samples.

\subsection{Explicit $C^1$ synthesis on a two-triangle patch}
Let the two triangles of a patch share the edge $\langle v_1,v_2\rangle$, with opposite vertices $v_3$ and $\widetilde v_3$.  Once the B-coefficients on the first triangle and the coefficients in the second triangle at domain points at least two rows away from the common edge are chosen, the edge row and first adjacent row on the second triangle are fixed by the smoothness relations of \Cref{subsec:c1edge}.  These free coefficients therefore give an explicit coordinate system for $S_p^1(P)$ of dimension $p^2+p+1$.  We denote the corresponding sparse synthesis matrix by $Z_P$ and write
\[
 c_P=Z_Pa_P.
\]
The kernel $\ker H_P$ is obtained directly from the explicit continuation map.  The numerical experiment in \Cref{sec:c1check} verifies directly that this explicit synthesis spans the same $C^1$ space as a rank-revealing nullspace computation to principal-angle accuracy near roundoff.

\subsection{Local quasi-Trefftz reduction and particular lift}
In the $C^1$ coordinates the projected residual moments form a matrix
\[
 C_P\in\C^{m_P\times n_P},\qquad n_P=\dim S_p^1(P).
\]
For a full-row-rank two-triangle patch, $m_P=p(p-1)$ and $n_P-m_P=2p+1$.  We use a full QR factorization of the adjoint,
\[
 C_P^*=Q_P^{\rm row}R_P,
\]
with the complete unitary factor retained.  The columns of the complementary block of $Q_P^{\rm row}$ form an orthonormal basis $W_P$ of $\ker C_P$.  The homogeneous quasi-Trefftz B-coefficient basis is therefore
\[
 Q_P=Z_PW_P.
\]
The inhomogeneous particular lift is obtained from the same residual equations by a stable least-squares solve.  Thus every patch field is represented as
\begin{equation}\label{eq:reducedpatch}
 c_P=c_P^f+Q_Pz_P,
 \qquad r_P:=\dim V_p^{\rm qT}(P).
\end{equation}
On every full-row-rank two-triangle patch in the experiments, $r_P=2p+1$.  A direct elimination based on a predetermined set of interior B-coefficients was also tested.  Its square pivot block deteriorates rapidly, with representative condition numbers growing from about $10^4$ at $p=6$ to about $10^{12}$ at $p=14$.  The adjoint-QR construction is therefore retained; in the same tests the reduced basis condition number remains about $1.4$--$1.6$ and the residual nullspace relation is satisfied at roundoff.

\subsection{Reference-triangle algebra and arbitrary affine meshes}
Reference Bernstein values, derivative maps, degree maps, and quadrature rules are precomputed once for each degree.  A physical triangle contributes only its affine Jacobian, inverse-transpose factors, barycentric gradients, edge orientation, and coefficient samples.  The $C^1$ synthesis map uses the barycentric coordinates of the opposite vertex and is therefore affine-local.  The Bernstein--B\'ezier smoothness machinery is affine-local and applies to arbitrary noncongruent triangles.

The adaptive coefficient projection is also triangulation independent.  The reference $L^2$ projection matrices are precomputed, and triangles that accept the same algebraic degree are processed in batches.  Since the affine Jacobian determinant occurs on both sides of the local projection equations, the same reference projection system applies on every triangle.  Residual moments are assembled in batches after the coefficient samples and geometry factors have been collected.

\subsection{Reduced global graph assembly}
Once \eqref{eq:reducedpatch} has been constructed, all expensive global operations remain in the reduced coordinates.  Interface value traces and conormal fluxes are evaluated directly in the $r_P$ reduced coordinates.  The global volume residual is similarly formed only after contraction with the reduced basis.  Moreover, the optimized volume channel tests the orthogonal complement of the degree-$(p-2)$ residual directions already annihilated locally.  That complement has dimension $2p+1$ on each triangle.  The low-order exact-coefficient contribution is precisely the AAD/source defect estimated in \Cref{prop:aadphysical}; hence this optimization follows the physical-norm decomposition directly.

If $R_h$ denotes the assembled graph-residual operator and $g_h$ the data vector, the discrete problem is
\begin{equation}\label{eq:ls}
 \min_z\norm{R_hz-g_h}_2^2,
 \qquad R_h^*R_hz=R_h^*g_h.
\end{equation}
Only neighboring patches share interface rows, so the normal matrix has the sparsity graph of the macro-patch adjacency graph.  The number of active global unknowns is
\begin{equation}\label{eq:dimension}
 N_{\rm active}=\sum_{P\in\Patches_h}r_P,
 \qquad
 N_{\rm active}=N_{\rm patch}(2p+1)
\end{equation}
for the full-row-rank two-triangle patches used in this paper.

\subsection{Equivalence and timing checks}
Each optimization was checked against the preceding unreduced implementation before being retained.  The explicit $C^1$ continuation and a rank-revealing QR produce local spline spaces whose largest principal angle is about $10^{-15}$, with smoothness residuals at roundoff.  At high degree the direct B-coefficient construction reduces the cost of the conformity stage by more than two orders of magnitude.

Batched coefficient projection reproduces the same selected approximation degree and coefficient error as triangle-by-triangle projection, while batched residual assembly and restriction to the surviving residual complement reduce the dominant local and volume work.  The archived benchmarks quantify the savings: explicit $C^1$ synthesis is about $100$--$170$ times faster than rank-revealing nullspace formation for $p=12$--$16$; batched global assembly is about $3.8$ times faster in representative $p=8$--$10$ cases; and restricting the volume channel to the surviving residual complement reduces that volume work by factors between $3.5$ and $5.8$.  In the exact-circle penetrable test at $\kappa=40$, $p=8$, and $200$ patches, optimized preprocessing takes $0.95$ s, global assembly $2.58$ s, and the sparse solve $1.24$ s for $3400$ active coordinates.  The archived equivalence tests compare each optimized stage with its unreduced predecessor; assembled matrices agree to roundoff-level relative differences, and the resulting solution vectors agree at approximately $10^{-14}$.  None of these optimizations uses congruent macro-patches; the same code path is used on all unstructured meshes below.

\section{Numerical experiments}\label{sec:num}
All global experiments use unstructured triangulations of the unit square $\Omega=(0,1)^2$.  Interior vertices are perturbed deterministically and the resulting point set is triangulated by Delaunay; adjacent triangles are then paired into nonoverlapping two-triangle macro-patches.  The local construction applies to the resulting noncongruent macro-patches.  In the resolution plots, $h$ denotes the maximum triangle diameter.  Because every macro-patch contains two adjacent triangles from the same shape-regular family, this quantity is uniformly equivalent to the patch diameter used in the analysis.  The thin lines in \Cref{fig:unstructured-patches} are triangle edges and the thick lines are macro-patch boundaries.  Within each colored patch the two polynomial pieces are joined strongly in $C^1$ by the Bernstein--B\'ezier smoothness conditions.  Across a thick patch boundary, neighboring patches communicate through the value trace and conormal flux in the graph residual.

The implementation used in this section is the triangulation-independent optimized code.  The $C^1$ synthesis map is formed explicitly from the B-coefficient continuation relation, coefficient projection is batched by Bernstein degree, residual moments are assembled in batches, the quasi-Trefftz kernel is obtained from a stable QR factorization of the adjoint local residual map, and global work is performed in the reduced $2p+1$ patch coordinates.  The global volume channel tests only the orthogonal complement of the residual directions already annihilated locally.  All reported errors are evaluated independently by higher-order quadrature.  The reported relative $L^2$ error is $\|u-u_h\|_{L^2(\Omega)}/\|u\|_{L^2(\Omega)}$, and the relative $H^1$ error is normalized analogously by the exact $H^1$ seminorm.

\begin{figure}[htbp]
\centering
\includegraphics[width=.60\linewidth]{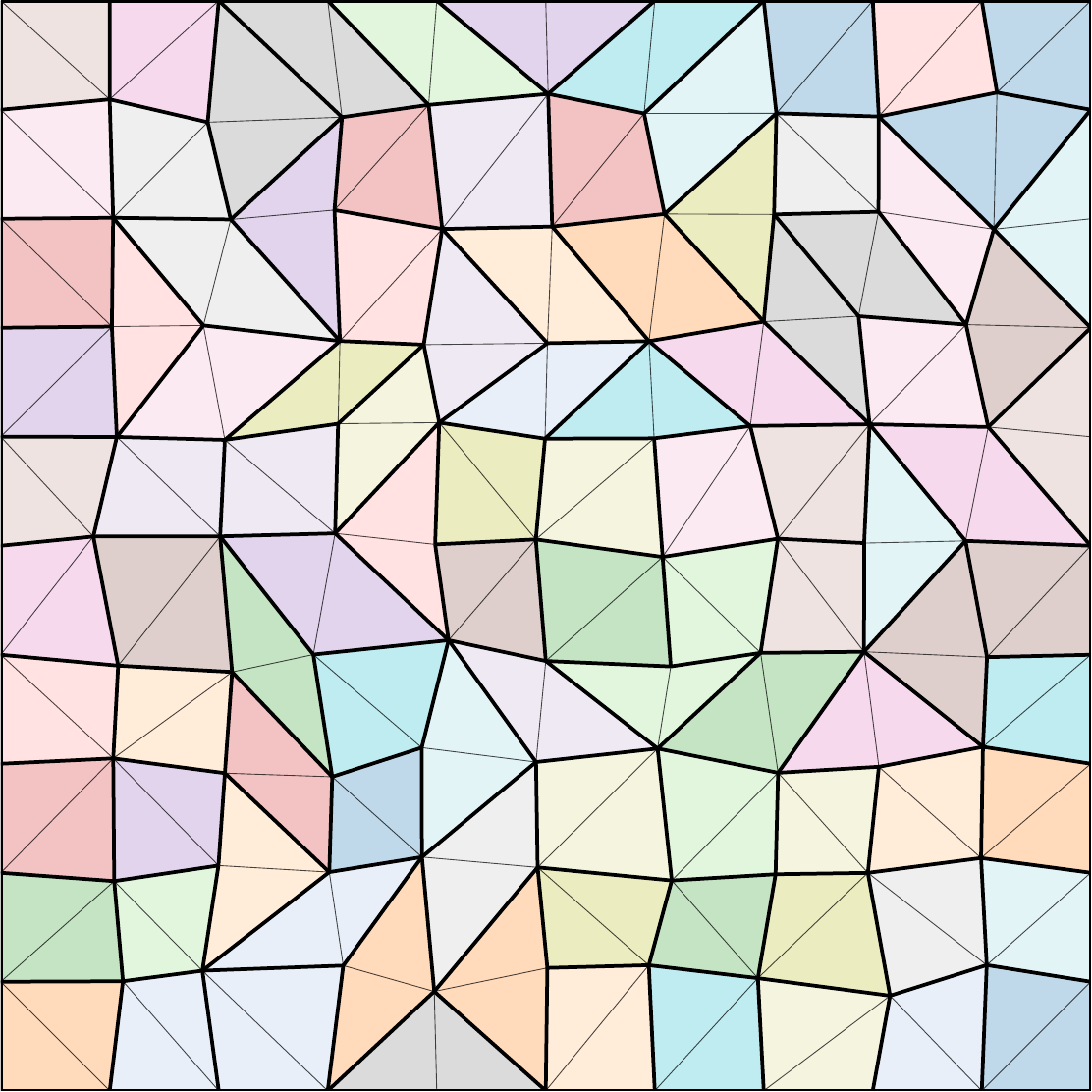}
\caption{Representative unstructured triangulation and two-triangle macro-patches used in the experiments.  Each shaded region is one macro-patch.  Thin lines are triangle edges; thick lines are patch boundaries.  The local space is strongly $C^1$ across the thin interior edge of each patch, while continuation across the thick interfaces is imposed weakly through the Cauchy graph residual.}
\label{fig:unstructured-patches}
\end{figure}

\subsection{Strong \texorpdfstring{$C^1$}{C1} conformity on noncongruent patches}\label{sec:c1check}
The first experiment checks the defining geometric property of the method before any approximation result is considered.  On every internal edge of every macro-patch we compare the two polynomial value traces and the two derivatives in the same canonical normal direction.  With the trace and normal-derivative maps defined in \Cref{subsec:c1matrix}, let $\varepsilon_0$ denote the relative value-trace mismatch and let $\varepsilon_1$ denote the corresponding derivative mismatch after multiplication by $h_e/p$.

\Cref{tab:c1unstruct} reports the worst mismatch over all patches of a representative unstructured mesh.  It also reports the range of nonzero-singular-value condition numbers of the scaled local smoothness matrix.  The complete $C^1$ spline space has the expected dimension $p^2+p+1$, and after the PDE reduction the quasi-Trefftz space has dimension $2p+1$ whenever the residual-moment map has full row rank.  Both the full spline basis and the reduced basis satisfy the smoothness conditions to roundoff.

\begin{table}[htbp]
\centering
\caption{Strong $C^1$ conformity and local smoothness conditioning on noncongruent two-triangle patches.  Errors are maxima over all patches in the mesh; the condition interval is the range of $\kappa_2^+(H_P)$.}
\label{tab:c1unstruct}
\begin{tabular}{rrrrrrr}
\toprule
$p$ & $\dim S_p^1(P)$ & max $\varepsilon_0$ & max $\varepsilon_1$ & max $\varepsilon_0(V_p^{\rm qT})$ & max $\varepsilon_1(V_p^{\rm qT})$ & condition range\\
\midrule
2  & 7   & $3.71\times10^{-16}$ & $4.05\times10^{-16}$ & -- & -- & $1.36$--$2.67$\\
4  & 21  & $2.50\times10^{-16}$ & $4.31\times10^{-16}$ & $2.84\times10^{-16}$ & $4.17\times10^{-16}$ & $1.37$--$2.79$\\
6  & 43  & $1.77\times10^{-16}$ & $2.01\times10^{-16}$ & $2.15\times10^{-16}$ & $2.90\times10^{-16}$ & $1.38$--$2.82$\\
8  & 73  & $1.80\times10^{-16}$ & $2.98\times10^{-16}$ & $2.16\times10^{-16}$ & $3.43\times10^{-16}$ & $1.38$--$2.83$\\
10 & 111 & $1.55\times10^{-16}$ & $2.10\times10^{-16}$ & $1.90\times10^{-16}$ & $3.11\times10^{-16}$ & $1.38$--$2.84$\\
12 & 157 & $1.51\times10^{-16}$ & $1.77\times10^{-16}$ & $1.46\times10^{-16}$ & $2.67\times10^{-16}$ & $1.38$--$2.84$\\
16 & 273 & $1.37\times10^{-16}$ & $2.35\times10^{-16}$ & -- & -- & $1.38$--$2.85$\\
\bottomrule
\end{tabular}
\end{table}

This experiment also clarifies the two different notions of continuation in the method.  Inside a macro-patch, equality of the edge B-coefficients fixes the common trace and the first-row smoothness relation fixes the transverse derivative, hence the local function is strongly $C^1$.  Across a macro-patch boundary, the neighboring traces are coupled through minimization of their value and conormal-flux mismatches.  The thick interfaces in \Cref{fig:unstructured-patches} therefore separate strong local spline conformity from weak global Cauchy coupling.

\subsection{Conditioning of the reduced global system}\label{sec:conditioning}
The local smoothness construction is very well conditioned, while the global least-squares matrix additionally contains propagation and interface effects.  If $R_h$ is the assembled graph-residual operator, the Hermitian system matrix is $K_h=R_h^*R_h$, so $\kappa_2(K_h)=\kappa_2(R_h)^2$.  For comparison we also form the diagonally equilibrated matrix $K_h^{\rm eq}=D^{-1/2}K_hD^{-1/2}$ with $D=\operatorname{diag}(K_h)$.  \Cref{tab:graphcond} reports representative unstructured cases.  The residual operator remains moderately conditioned in the tested range, and the simple diagonal scaling reduces the normal-matrix condition in every displayed case.

\begin{table}[htbp]
\centering
\caption{Conditioning of the graph-residual system on unstructured macro-patch meshes.}
\label{tab:graphcond}
\begin{tabular}{rrrrrrr}
\toprule
$\kappa$ & patches & $p$ & active & $\kappa_2(K_h)$ & $\kappa_2(K_h^{\rm eq})$ & $\kappa_2(R_h)$\\
\midrule
20 & 1  & 4  & 9   & $9.42$ & $8.42$ & $3.07$\\
20 & 1  & 8  & 17  & $6.78\times10^1$ & $6.48\times10^1$ & $8.24$\\
20 & 9  & 8  & 153 & $3.26\times10^2$ & $2.75\times10^2$ & $18.1$\\
20 & 16 & 8  & 272 & $2.13\times10^3$ & $8.92\times10^2$ & $46.2$\\
20 & 25 & 10 & 525 & $1.90\times10^3$ & $1.60\times10^3$ & $43.6$\\
40 & 25 & 8  & 425 & $1.37\times10^3$ & $9.94\times10^2$ & $37.0$\\
\bottomrule
\end{tabular}
\end{table}

\subsection{A matrix-valued heterogeneous medium and unstructured refinement}\label{sec:hf}
The remaining experiments use one fixed smooth anisotropic medium and one fixed oscillatory exact solution, so every part of the method is active simultaneously: local $C^1$ synthesis, coefficient projection for basis construction, conormal fluxes with a matrix field, the affine particular lift, and the reduced global graph residual.  The coefficient field is the symmetric positive definite matrix
\[
A(x,y)=\begin{pmatrix}
1.4+0.25\sin(2\pi x)\cos(\pi y) & 0.12\sin(\pi x)\sin(\pi y)\\
0.12\sin(\pi x)\sin(\pi y) & 1.2+0.20\cos(\pi x)\sin(2\pi y)
\end{pmatrix},
\]
with scalar index
\[
\eta(x,y)=1.10+0.15e^{0.4x-0.3y}+0.10\sin(\pi xy),
\]
and the manufactured wave field
\[u(x,y)=e^{i\kappa\phi(x,y)},
\]
with curved phase
\[
\phi(x,y)=x+0.35y+0.05\sin(2\pi x)\sin(\pi y).
\]
This example exercises a genuinely nonpolynomial, matrix-valued heterogeneous medium with spatially varying anisotropy.

\begin{figure}[htbp]
\centering
\includegraphics[width=.98\linewidth]{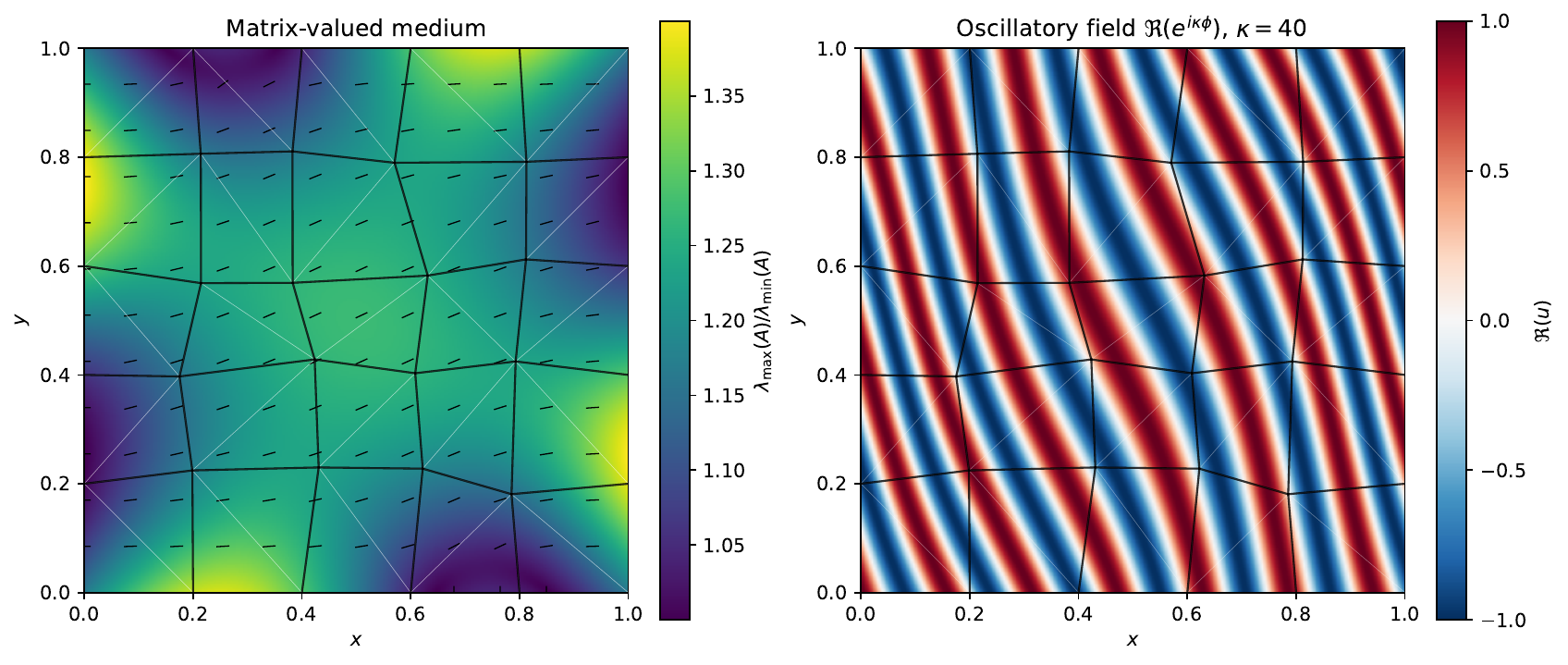}
\caption{The matrix-valued heterogeneous test problem used in the remaining experiments.  Left: anisotropy ratio $\lambda_{\max}(A)/\lambda_{\min}(A)$ with principal directions overlaid, together with the unstructured macro-patch boundaries.  Right: planar rendering of the manufactured exact oscillatory field $\Re(e^{i\kappa\phi(x,y)})$ at $\kappa=40$ on the same domain, using a centered blue--white--red colormap on the unstructured macro-patch geometry.}
\label{fig:matrix-medium}
\end{figure}

\Cref{fig:matrix-medium} makes the geometry of the example explicit.  The left panel shows that the coefficient field rotates and changes strength across the domain, while the right panel shows the curved wave pattern that drives the forcing and impedance data.  The refinement experiment in \Cref{fig:unstructured-refinement} solves exactly this matrix-valued problem with the optimized reduced-coordinate implementation and varies only the unstructured mesh.  The active dimension remains exactly $17$ per patch.

\begin{figure}[htbp]
\centering
\includegraphics[width=.62\linewidth]{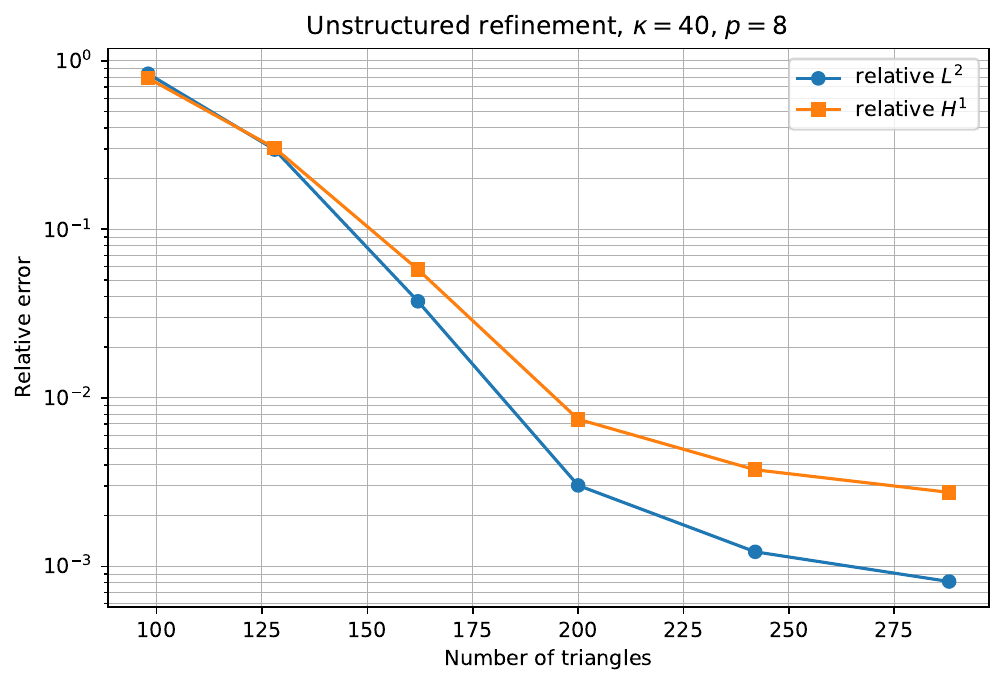}
\caption{Refinement on the unstructured macro-patch family for $\kappa=40$ and $p=8$ in the matrix-valued heterogeneous medium of \Cref{fig:matrix-medium}.}
\label{fig:unstructured-refinement}
\end{figure}

\begin{table}[htbp]
\centering
\caption{Representative points from the unstructured refinement run at $\kappa=40$, $p=8$.}
\label{tab:unstruct-refine}
\begin{tabular}{rrrrr}
\toprule
triangles & patches & active & relative $L^2$ & core time (s)\\
\midrule
128 & 64  & 1088 & $2.97\times10^{-1}$ & 0.48\\
162 & 81  & 1377 & $3.74\times10^{-2}$ & 0.70\\
200 & 100 & 1700 & $3.02\times10^{-3}$ & 0.66\\
242 & 121 & 2057 & $1.22\times10^{-3}$ & 0.97\\
288 & 144 & 2448 & $8.13\times10^{-4}$ & 1.07\\
\bottomrule
\end{tabular}
\end{table}

The convergence persists across noncongruent patches with varying triangle shapes, while the local dimension and the $C^1$ construction remain unchanged.

\subsection{Fixed-mesh degree resolution and high wavenumber}\label{sec:presolution}
To separate mesh refinement from local polynomial resolution, we next fix one unstructured mesh with 144 macro-patches and vary only $p$.  \Cref{fig:fixed-h-p} extends the degree sweep well beyond the initial range: the $\kappa=40$ and $\kappa=120$ curves are continued through $p=15$, while the $\kappa=80$ curve is continued through $p=16$.  At $\kappa=40$ the error decreases rapidly once the method enters the resolved regime and reaches $7.04\times10^{-8}$ at $p=15$.  At $\kappa=80$ a clear transition appears near $\kappa h/p\approx1$: the error initially remains large, then drops sharply to $1.38\times10^{-3}$ at $p=16$.  At $\kappa=120$ the fixed mesh remains underresolved through $p=15$, so the high-wavenumber curve records the persistence of the resolution barrier on this mesh.

\begin{figure}[htbp]
\centering
\includegraphics[width=.62\linewidth]{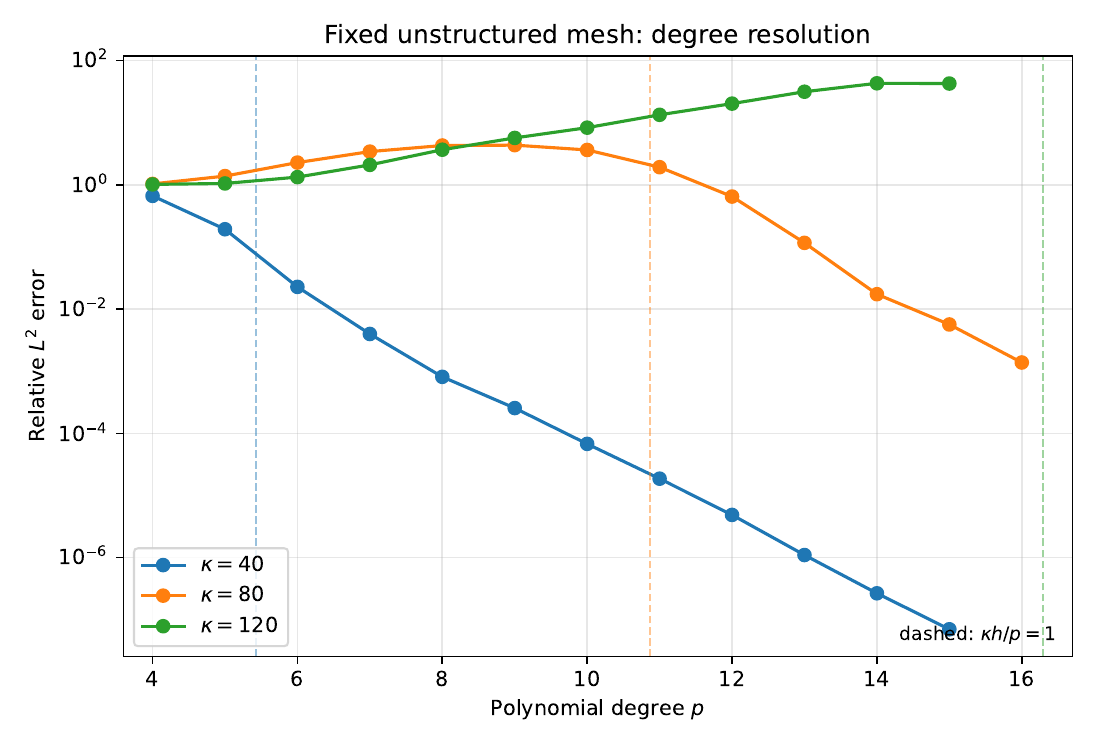}
\caption{Relative $L^2$ error versus polynomial degree on one fixed unstructured mesh with 144 macro-patches.  The degree sweep extends through $p=15$ for $\kappa=40$ and $120$ and through $p=16$ for $\kappa=80$.  Dashed vertical lines mark $\kappa h/p=1$ for each wavenumber; the change in slope tracks the local resolution threshold.}
\label{fig:fixed-h-p}
\end{figure}

Higher degrees substantially improve the high-frequency approximation after the local resolution threshold is crossed.  On the fixed 144-patch mesh at $\kappa=80$, increasing the degree from $p=10$ to $16$ reduces the relative $L^2$ error from $3.65$ to $1.38\times10^{-3}$ while the active dimension grows from $3024$ to $4752$.  The $\kappa=40$ curve continues its near-geometric decay through $p=15$, whereas the $\kappa=120$ curve remains unresolved through the largest degree reached on this mesh.  A separate finer-mesh computation at $\kappa=120$ with 256 patches and $p=16$ gives relative $L^2$ error $5.79\times10^{-2}$ with 8448 active coordinates, confirming that the remaining limitation is resolution rather than a failure of the local reduction.

\subsection{Pollution at fixed resolution}\label{sec:pollution}
To probe pollution separately from simple under-resolution, approximation resolution must remain approximately fixed while wavenumber increases.  We therefore fix $p$, choose a target value of $\kappa h/p$, and vary the unstructured mesh with $\kappa$ so that the achieved resolution parameter remains close to that target.  We use ``pollution'' here for frequency-dependent error growth observed along such fixed-resolution sequences, distinct from the large approximation error of an underresolved single mesh.  The complete computational sweep contains four target levels and ten wavenumbers per level; the main paper shows only the two informative regimes, $\kappa h/p\approx0.5$ and $0.8$.  The intermediate levels are retained in the accompanying data archive.

\begin{figure}[htbp]
\centering
\includegraphics[width=.90\linewidth]{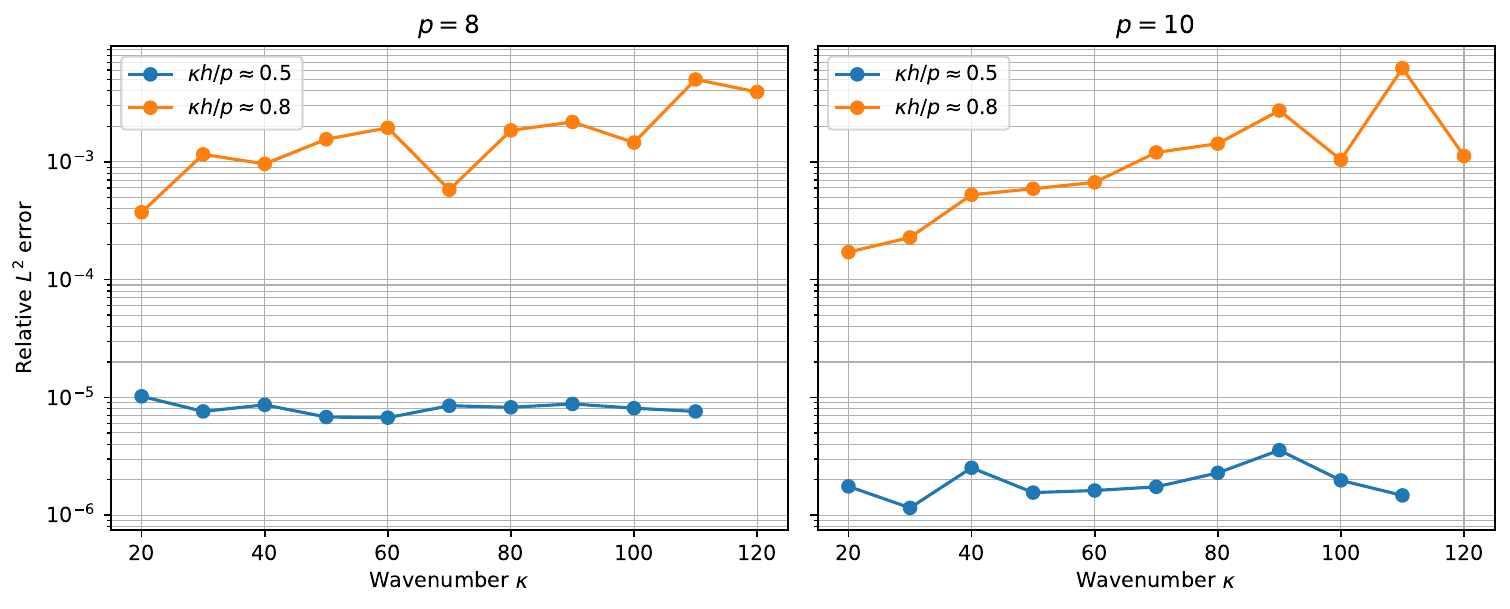}
\caption{Pollution study on unstructured meshes.  Each panel fixes the polynomial degree and compares a well-resolved line, $\kappa h/p\approx0.5$, with a near-threshold line, $\kappa h/p\approx0.8$.  Ten wavenumbers from $20$ through $110$ are used on every line; the $0.8$ lines are additionally extended to $\kappa=120$.}
\label{fig:pollution}
\end{figure}

The fitted power laws are used only as descriptive summaries of these finite sweeps.  At $\kappa h/p\approx0.5$ the errors remain essentially flat: from $\kappa=20$ to $110$ the endpoint error ratio is $0.74$ for $p=8$ and $0.84$ for $p=10$.  At $\kappa h/p\approx0.8$ the corresponding ratios are $13.4$ and $36.4$, respectively; descriptive log--log fits give exponents $0.95$ and $1.78$.  Thus the near-threshold lines show clear frequency-dependent growth, while the well-resolved lines remain essentially flat.  The experiment provides a practical separation between the resolved and polluted regimes; the fitted exponents serve as finite-range descriptive summaries.

The archived data also contain the intermediate targets $\kappa h/p\approx0.6$ and $0.7$.  Their fitted exponents are close to zero at $0.6$ ($-0.03$ for $p=8$ and $-0.18$ for $p=10$) and become positive at $0.7$ ($0.28$ and $0.65$, respectively), placing $0.7$ in the transition between the resolved and near-threshold regimes.  They are omitted from the main figure to keep the comparison focused.

\subsection{Penetrable scattering with an exact NtD truncation}\label{sec:penetrable-ntd}
To complement the manufactured heterogeneous example, we also consider a genuine scattering computation in the exact disk $\Omega=B_R$ with $R=1/2$.  The total field satisfies
\[
-\Delta u-\kappa^2\eta(x,y)u=0\qquad\text{in }\Omega,
\]
and is driven by the incident plane wave
\[u^{\mathrm{inc}}(x,y)=e^{i\kappa x}.\]
The refractive index is the smooth off-center inclusion
\[
\eta(x,y)=1+2.4\exp\!\left(-\left(\frac{\sqrt{(x-0.035)^2+1.45(y+0.015)^2}}{0.16}\right)^8\right),
\]
which produces visible refraction, focusing, and asymmetric interference in the scattered field.  The circular truncation boundary is kept exact, and the outgoing condition for the scattered field is imposed through the Fourier Neumann-to-Dirichlet map on $\Gamma_R$; see \cite{KapitaThesis2016,KapitaMonk2018,MonkPenaSelgas2025}.  The assembly is carried out directly on the exact disk, so this test contains no geometric boundary approximation error.

\begin{figure}[htbp]
\centering
\includegraphics[width=.96\linewidth]{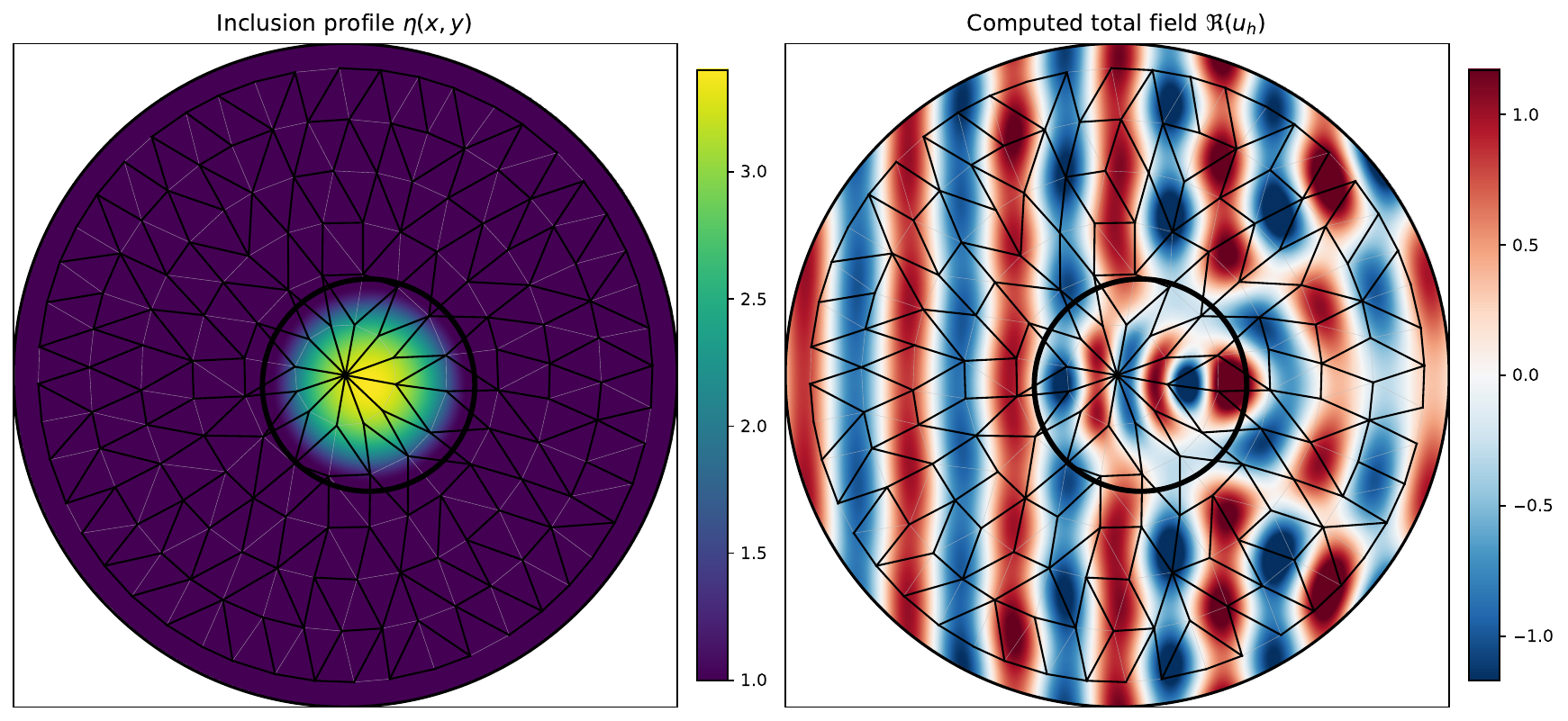}
\caption{Penetrable scattering in the exact disk with an exact Fourier NtD truncation at $\kappa=40$ and $p=8$.  Left: smooth off-center inclusion profile $\eta(x,y)$.  Right: computed total field $\Re(u_h)$ on the same exact circle, showing refraction and interference generated by the inhomogeneity.  The macro-patch boundaries are overlaid.}
\label{fig:penetrable-ntd}
\end{figure}

\begin{table}[htbp]
\centering
\caption{Timing data for the penetrable scattering experiment of \Cref{fig:penetrable-ntd}.}
\label{tab:penetrable-ntd}
\begin{tabular}{rrrrrrr}
\toprule
$\kappa$ & $p$ & patches & active & prep. (s) & assembly (s) & solve (s)\\
\midrule
40 & 8 & 200 & 3400 & 0.95 & 2.58 & 1.24\\
\bottomrule
\end{tabular}
\end{table}

The corresponding total runtime is $4.77$ s, and the relative NtD boundary residual is $1.04\times10^{-3}$.  The preprocessing uses the batched conformity-first implementation: explicit $C^1$ continuation, batched coefficient projection, batched residual-moment formation, and local QR reduction.  Every patch again retained the trace-sized dimension $2p+1=17$.  \Cref{fig:penetrable-ntd} shows that the same conformity-first reduced space captures a physically structured heterogeneous scattering pattern on an unstructured macro-patch mesh while preserving the exact circular geometry at the truncation boundary.

\subsection{Sensitivity to the local coefficient projection}\label{sec:coeftol}
The local quasi-Trefftz coordinates use an adaptive Bernstein approximation of the nonpolynomial medium, whereas the global residual and error evaluation use the original coefficient functions.  To quantify the influence of this auxiliary approximation, we repeat the resolved $\kappa=40$, $p=8$ calculation on the 144-patch unstructured mesh while varying the local coefficient-projection tolerance.

\begin{table}[htbp]
\centering
\caption{Coefficient-projection sensitivity for $\kappa=40$, $p=8$, and 144 unstructured macro-patches.}
\label{tab:coeftol}
\begin{tabular}{rrrr}
\toprule
tolerance & maximum degree & maximum projection error & relative $L^2$ error\\
\midrule
$10^{-3}$ & 2 & $8.19\times10^{-5}$ & $8.78\times10^{-4}$\\
$10^{-5}$ & 4 & $8.54\times10^{-6}$ & $8.13\times10^{-4}$\\
$10^{-7}$ & 6 & $9.97\times10^{-8}$ & $8.13\times10^{-4}$\\
$10^{-9}$ & 6 & $1.73\times10^{-10}$ & $8.13\times10^{-4}$\\
\bottomrule
\end{tabular}
\end{table}

Once the coefficient approximation is below the discretization error, tightening the local tolerance leaves the field unchanged to the displayed digits.  The saturation near $8.13\times10^{-4}$ is therefore the spatial discretization error on this fixed mesh, rather than a coefficient-projection or quadrature floor.  The experiments use the coefficient-space test \eqref{eq:coefftol} because it is inexpensive, reusable across wavenumbers, and directly available during batched medium projection.  The operator-level stopping test \eqref{eq:aadstop} is more selective when the retained singular gap is small and is intended for regimes in which the local residual map approaches rank sensitivity.

\subsection{An Airy turning-point transition beyond the uniformly positive regime}\label{sec:airy}
As a final stress test, we consider a one-dimensional turning-point field embedded in the two-dimensional domain.  This example lies intentionally outside the uniformly positive-index assumptions used in \Cref{sec:physical}.  We take
\[
A(x,y)=I,
\qquad
\eta(x,y)=x-x_\star,
\qquad
x_\star=0.53,
\]
and the exact solution
\[u(x,y)=\operatorname{Ai}\!\bigl(-\kappa^{2/3}(x-x_\star)\bigr),
\]
where $\operatorname{Ai}$ denotes the Airy function.  Since $\operatorname{Ai}''(z)=z\operatorname{Ai}(z)$, this gives the model problem
\[
-\Delta u-\kappa^2(x-x_\star)u=0.
\]
The line $x=x_\star$ is a turning point: the exact field is evanescent on one side and oscillatory on the other.  For this sign-changing coefficient, the boundary least-squares term uses a fixed positive Robin weight on $\partial\Omega$.

\begin{figure}[htbp]
\centering
\includegraphics[width=.82\linewidth]{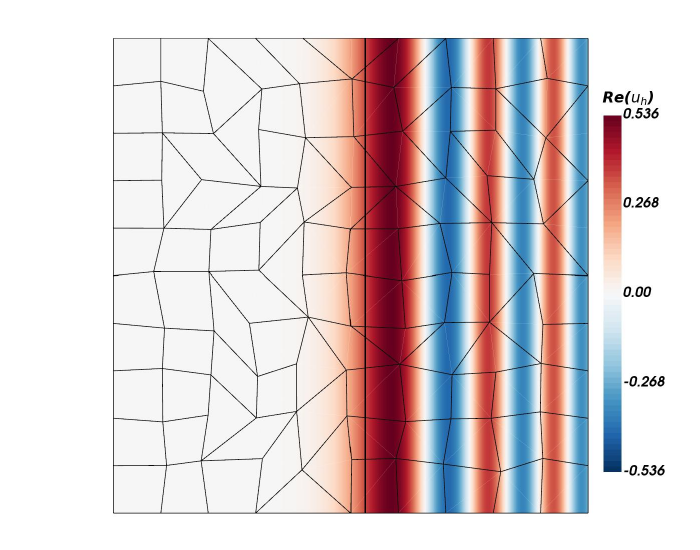}
\caption{Computed Airy turning-point solution on the unstructured macro-patch mesh for $\kappa=80$, $p=10$, and $100$ macro-patches.  The thick vertical line marks the turning point $x=x_\star$.  The centered blue--white--red colormap emphasizes the transition from the evanescent region to the oscillatory region.}
\label{fig:airy}
\end{figure}

\Cref{fig:airy} shows that the conformity-first Bernstein reduced space crosses the turning point cleanly using the same polynomial construction throughout the domain.  The oscillatory side is resolved with the same $2p+1$ trace-sized coordinates per patch that arise in the positive-index experiments.  Representative errors are listed in \Cref{tab:airy}; every patch retained the full trace-sized ranks $13$, $17$, and $21$ for $p=6$, $8$, and $10$, respectively.

\begin{table}[htbp]
\centering
\caption{Airy turning-point stress test on unstructured macro-patches.}
\label{tab:airy}
\begin{tabular}{rrrrrr}
\toprule
$\kappa$ & mesh parameter $n$ & $p$ & active & relative $L^2$ & relative $H^1$\\
\midrule
20 & 4  & 6  & 208  & $1.16\times10^{-4}$ & $6.54\times10^{-4}$\\
20 & 4  & 8  & 272  & $1.32\times10^{-6}$ & $1.05\times10^{-5}$\\
40 & 6  & 8  & 612  & $4.18\times10^{-5}$ & $1.92\times10^{-4}$\\
40 & 6  & 10 & 756  & $1.52\times10^{-6}$ & $9.98\times10^{-6}$\\
80 & 10 & 10 & 2100 & $5.61\times10^{-5}$ & $2.56\times10^{-4}$\\
\bottomrule
\end{tabular}
\end{table}

The positive-index theory of \Cref{sec:physical} applies to the uniformly elliptic positive-index regime.  The Airy experiment extends the numerical study to a sign-changing index and shows that the same local quasi-Trefftz mechanism remains effective as the solution changes character across the domain.  The unstructured macro-patch construction retains its trace-sized coordinates through the turning line and captures the transition at the accuracy reported in \Cref{tab:airy}.

\section{Relation to Trefftz methods and computational implications}\label{sec:related}
The proposed approximation is a discrete near-kernel inside a conforming polynomial space. Exact Trefftz functions satisfy the Helmholtz equation identically in each element, whereas the present local constraint annihilates all residual moments visible in the degree-$(p-2)$ test space. On a full-row-rank two-triangle patch, $p(p-1)$ interior residual directions are removed from the $p^2+p+1$ coordinates of the strongly $C^1$ spline space, leaving $2p+1$ reduced coordinates. The fixed-mesh experiments show why this compression is useful computationally: once the local resolution threshold is crossed, increasing $p$ improves the wave approximation rapidly while the globally coupled dimension grows only linearly with $p$.

The construction differs from Taylor quasi-Trefftz spaces in where the PDE information is imposed. Taylor methods annihilate a finite jet of the residual near a point \cite{ImbertGerard2025Taylor}; here projected residual moments are imposed over the whole macro-patch. Embedded Trefftz methods also extract PDE-adapted polynomial spaces algebraically \cite{LehrenfeldStocker2023,StockerVoulis2026}, but the present reduction is carried out only after the Bernstein pieces have been restricted to the geometrically admissible $C^1$ space. This is the source of the conformity-first structure.

Plane-wave DG, UWVF, and related Trefftz-DG methods couple exact local Helmholtz solutions primarily through boundary or impedance traces \cite{CessenatDespres1998,HiptmairMoiolaPerugia2011,BuffaMonk2008}. The present global coupling is similar in spirit, but the local representation is polynomial and the remaining consistency defect is retained explicitly in the graph functional. This allows smooth heterogeneous coefficients and nonzero forcing to be treated without choosing propagation directions or replacing the local basis by special functions.

From an implementation viewpoint, the principal tradeoff is clear. The local construction temporarily works with a quadratic-size conforming polynomial space, while the sparse global problem uses only the linear-size reduced coordinates. This makes preprocessing the natural location for dense local algebra. Explicit $C^1$ continuation, batching, and QR reduction keep that cost small in the tested range, and geometry-dependent quantities can be cached across frequency sweeps on a fixed mesh. The global matrix inherits the macro-patch adjacency graph and therefore retains sparse nearest-neighbor coupling.

The continuous constants in the physical-energy estimate remain problem dependent. In particular, the impedance stability and Cauchy-lifting constants can depend on wavenumber, anisotropy, and the domain. The analysis therefore separates these continuous stability factors from the discrete $\kappa h/p$ dependence rather than claiming a wavenumber-uniform heterogeneous bound. The fixed-resolution experiments provide a direct computational probe of the discrete part of this mechanism.

The present implementation uses two-triangle macro-patches. For a general multi-triangle patch, the same algorithm first assembles the smoothness matrix $H_P$ and then forms the residual map $C_P$ in the resulting conforming coordinates. The dimensions are $N_TN_p-\operatorname{rank}H_P$ before PDE reduction and $N_TN_p-\operatorname{rank}H_P-\operatorname{rank}C_P$ afterward. Larger patches can therefore exhibit topology-dependent smoothness dependencies or residual-rank loss. Efficient sparse continuation on cyclic aggregates and a quantitative study of patch-shape effects are natural extensions.

A second extension concerns the adaptive algebraic-degree stopping criterion. The experiments use the coefficient-space criterion because it is inexpensive and robust when the coefficient approximation is comfortably below the discretization error. The operator-level criterion in \eqref{eq:aadstop} becomes preferable when the retained singular gap is small, because it measures the perturbation on the residual map directly. This suggests a future implementation in which coefficient degree and local kernel tolerance are selected together.

\section{Conclusions}\label{sec:conclusions}
We introduced a computational framework in which a strongly $C^1$ Bernstein macro-patch is compressed locally to a quasi-Trefftz space before global coupling. The conformity-first ordering is the key structural choice: geometry determines the admissible spline coordinates, the PDE removes resolvable residual directions inside that space, and the global graph residual couples only the resulting trace-sized coordinates. For full-row-rank two-triangle patches the local dimension is reduced from $p^2+p+1$ to $2p+1$.

The same architecture separates heterogeneous coefficient approximation from the global discretization. Projected coefficients are used to construct the local reduced space, while the original medium enters the global residual. The perturbation analysis quantifies the resulting kernel rotation through the retained singular gap, and the physical-energy estimate separates the represented high-order residual, the low-order coefficient/source defect, and the unresolved residual tail.

The numerical results show that the reduction is effective on genuinely noncongruent unstructured patches. Strong $C^1$ conformity is maintained to roundoff, the local reduced basis remains well conditioned, and high-order convergence appears once the local resolution threshold is crossed. Fixed-resolution experiments distinguish a regime with little frequency growth from a near-threshold polluted regime. The exact-circle penetrable-scattering test demonstrates that the reduced polynomial representation can reproduce refraction and interference patterns while using a Fourier NtD map on the true circular boundary. The Airy experiment shows that the same local algebra remains stable through a turning-point transition outside the positive-index assumptions used in the physical-energy estimate.

Computationally, the method concentrates dense work in local preprocessing and keeps the global system in reduced coordinates. With explicit $C^1$ continuation and batched local algebra, preprocessing is no longer the dominant cost in the representative scattering experiment. This separation makes the approach attractive for repeated solves, parameter sweeps, and future adaptive extensions in which local degree, coefficient approximation, and patch topology are varied independently.

\section*{Data and code availability}
The accompanying source archive contains the deterministic unstructured mesh generator, optimized local reduction and graph-assembly routines, exact-circle NtD scattering driver, Airy test, degree-sweep driver, figure scripts, and the tabulated data used in the numerical section.

\section*{Declaration of generative AI and AI-assisted technologies in the writing process}
During the preparation of this work, the author used ChatGPT (OpenAI) for code prototyping and testing, numerical experimentation, visualization support, mathematical checking, and drafting and editing assistance. The author reviewed and verified the mathematical arguments, numerical results, references, and final manuscript and takes full responsibility for the content of the article.

\printbibliography
\end{document}